\documentclass[12pt]{article} 

\usepackage{float}
\usepackage{tikz-cd}
\usepackage{amsxtra}
\usepackage{amsmath}
\usepackage{amssymb}
\usepackage{amsfonts}
\usepackage[all,2cell]{xy}
\usepackage{amsthm}
\usepackage{enumitem}
\usepackage{hyperref}
\usepackage{subfigure}
\usepackage{euscript}
\usepackage[margin=3cm]{geometry}
\usepackage[bbgreekl]{mathbbol}

\numberwithin{equation}{section}

\newtheorem{thm}[equation]{Theorem}
\newtheorem*{thm*}{Theorem}
\newtheorem{cor}[equation]{Corollary}
\newtheorem{lem}[equation]{Lemma}
\newtheorem{prop}[equation]{Proposition}

\theoremstyle{definition}
\newtheorem{defi}[equation]{Definition}
\newtheorem{rem}[equation]{Remark}
\newtheorem{exa}[equation]{Example}
\newtheorem*{exa*}{Example}

\usepackage{xcolor}%
\hypersetup{%
    colorlinks,%
    linkcolor={red!50!black},%
    citecolor={blue!50!black},%
    urlcolor={blue!80!black}%
}

\def\on{\operatorname}
\def\op{{\on{op}}}

\def\id{\on{id}}

\def\im{\on{im}}
\def\Fun{\on{Fun}}
\def\lim{\on*{lim}}

\def\tot{\on*{tot}}

\def\Hom{\on{Hom}}
\def\Map{\on{Map}}

\def\N{\on{N}}
\def\Nsc{\on{N^{\on{sc}}}}
\def\P{\on{P}}
\def\L{\on{L}}

\def\hra{\hookrightarrow}
\def\lra{\longrightarrow}
\def\lla{\longleftarrow}
\def\llra{\longleftrightarrow}

\def\CC{\mathbb{C}} 
\def\DD{\mathbb{D}} 
 
\def\ZZ{\mathbb{Z}}

\def\CChi{\mathbb{\bbchi}}
\DeclareMathSymbol\DDelta\mathord{bbold}{"01}
\DeclareMathSymbol\GGamma\mathord{bbold}{"00}
\DeclareMathSymbol\SSigma\mathord{bbold}{'117}

\def\Nc{\EuScript{N}}
\def\C{\on{C}}
\def\Cc{\EuScript{C}}
\def\Fc{\EuScript{F}}

\def\Pc{\EuScript{P}}

\def\Ab{\on{Ab}}
\def\Cat{\on{Cat}}
\def\CCat{\on{\mathbb{C}at}}
\def\Grp{\on{Grp}}
\def\Fib{\on{Fib}}
\def\St{{\EuScript S}t}
\def\Set{{\on{Set}}}
\def\sSet{ {\Set_{\Delta}}}
\def\msSet{ {\Set^{+}_{\Delta}}}
\def\Ch{\on{Ch}}

\def\NN{{\mathbb{N}}}

\def\A{{\EuScript A}}
\def\X{{\EuScript X}}
\def\Y{{\EuScript Y}}
\def\B{{\EuScript B}}

\def\M{{\EuScript M}}

\def\M{ {\EuScript M}}

\setlist[enumerate,1]{label=(\arabic{*})}
\setlist[enumerate,2]{label=(\alph{*})}
\setlist[enumerate,3]{label=(\roman{*})}

\title{A categorified Dold-Kan correspondence} 
\author{Tobias Dyckerhoff \footnote{Hausdorff Center for Mathematics, 
				Endenicher Allee 62, 
				53115 Bonn, 
				Germany, 
				email: {\tt dyckerho@math.uni-bonn.de}}}

\begin{document}

\maketitle

\begin{abstract}
	In this work, we establish a categorification of the classical Dold-Kan correspondence in
	the form of an equivalence between suitably defined $\infty$-categories of simplicial stable
	$\infty$-categories and connective chain complexes of stable $\infty$-categories. The result
	may be regarded as a contribution to the foundations of an emerging subject that could be
	termed categorified homological algebra.
\end{abstract}

\tableofcontents

\section{Introduction}

A central tool in classical homological algebra is the construction of a chain complex from a
simplicial abelian group via the formula
\[
	d = \sum_{i=0}^n (-1)^i d_i.
\]
The fact that a large number of interesting complexes arise via this procedure is not a coincidence
-- the classical Dold-Kan correspondence \cite{dold:homology,kan:functors} states that the passage
to normalized chains establishes an equivalence of categories:

\begin{thm*}[Dold,Kan] The normalized chains functor $\C$ provides an equivalence 
\[
	\C: \Ab_{\Delta} \overset{\simeq}{\llra} \Ch_{\ge 0}(\Ab): \N
\]
between the category $\Ab_{\Delta}$ of simplicial abelian groups and the category $\Ch_{\ge 0}(\Ab)$
of connective chain complexes with inverse given explicitly by the Dold-Kan nerve $\N$.
\end{thm*}

In this work, we establish a categorification of the classical Dold-Kan correspondence where the
category $\Ab$ of abelian groups gets replaced by the $(\infty,2)$-category $\St$ of stable
$\infty$-categories. 

\begin{thm*} The categorified normalized chains functor $\Cc$ furnishes an equivalence 
\[
	\Cc: \St_{\DDelta} \overset{\simeq}{\llra} \Ch_{\ge 0}(\St): \Nc
\]
between the $\infty$-category $\St_{\DDelta}$ of $2$-simplicial stable $\infty$-categories and the
$\infty$-category $\Ch_{\ge 0}(\St)$ of connective chain complexes of stable $\infty$-categories
with explicit inverse given by the categorified Dold-Kan nerve $\Nc$.
\end{thm*}

We refer the reader to \S 3.1 and \S 3.2 for an explication of the terminology and a precise
statement of the theorem. A key ingredient of the proof is an explicit construction of the {\em
categorified Dold-Kan nerve} $\Nc$. Its classical counterpart $\N$ associates to a chain complex
$B_{\bullet}$ of abelian groups the simplicial abelian group $\N(B_{\bullet})$ which can be
described as follows: the group of $n$-simplices is given by collections $\{x_{\sigma}\}$,
parametrized by monotone maps $\sigma: [k] \hra [n]$, of elements $x_{\sigma} \in B_k$
subject to the equations
\[
	d(x_{\sigma}) = \sum_{i=0}^k (-1)^i x_{\sigma \circ \partial_i}.
\]

\begin{exa*}
A $2$-simplex in $\N(B_{\bullet})$ consists of 
\begin{itemize}
	\item elements $x_{012} \in B_2$, $x_{01},x_{02},x_{12}
\in B_1$, and $x_0,x_1,x_2 \in B_0$, 
	\item satisfying the equations 
		\begin{enumerate}
			\item $d(x_{012}) = x_{12} - x_{02} + x_{01}$,
			\item $d(x_{ij}) = x_{j} - x_{i}$.
		\end{enumerate}
\end{itemize}
As a foretaste, we provide an informal description of the data comprising a $2$-simplex in the
categorified Dold-Kan nerve $\Nc(\B_{\bullet})$ associated to a complex 
\[
	\B_0 \overset{d}{\lla} \B_1 \overset{d}{\lla}\B_2 \overset{d}{\lla} \cdots
\]
of stable $\infty$-categories:
\begin{itemize}
	\item objects $X_{012} \in \B_2$, $X_{01},X_{02},X_{12} \in \B_1$, and $X_0,X_1,X_2 \in
		\B_0$, 
	\item a chain of morphisms $X_{0} \to X_{1} \to X_{2}$ in $\B_0$, 
	\item a $3$-term complex 
		\[
		\begin{tikzcd}  
			X_{01} \arrow{r}\arrow{d}& X_{02} \arrow{d}\\
		    	0 \arrow{r}& X_{12}
		\end{tikzcd}
		\]
		in $\B_1$, 
	\item together with coherent data that exhibits 
		\begin{enumerate}
			\item $d(X_{012})$ as the totalization of the complex $X_{01} \to X_{02} \to
				X_{12}$,
			\item $d(X_{ij})$ as the cofiber of $X_i \to X_j$.
		\end{enumerate}
\end{itemize}
Note that the data of a $2$-simplex in $\Nc(\B_{\bullet})$ defines, upon passage to classes in the
respective Grothendieck groups, a $2$-simplex in $\N(K_0(\B_{\bullet}))$. This observation
generalizes to simplices in all dimensions and is a justification for the use of the term {\em
categorification}.
\end{exa*}

The categorified Dold-Kan nerve $\Nc$ unifies various known constructions from algebraic $K$-theory:
\begin{enumerate}[label=(\Roman{*})]
	\item Let $\B$ be a stable $\infty$-category and let $\B[1]$ denote the chain complex
		\[
			0 \lla \B \lla 0 \lla 0 \lla \dots
		\]
		concentrated in degree $1$. Then $\Nc(\B[1])$ is an $\infty$-categorical version
		of Waldhausen's $S_{\bullet}$-construction (cf. \cite{waldhausen}). 
		Waldhausen's $S_{\bullet}$-construction is usually considered as a simplicial object
		in the category $\Cat$ (or $\Cat_{\infty}$). While the additional $2$-functoriality present in our
		treatment does not seem to appear explicitly in the literature, it {\em does} feature implicitly,
		yet crucially, in Waldhausen's proof of the additivity theorem \cite{waldhausen} and its
		modification for stable $\infty$-categories presented in \cite{lurie:ktheory}. This observation
		will be explored in detail elsewhere. 

	\item Let $f:\B_1 \to \B_0$ be an exact functor of stable $\infty$-categories. Then applying
		$\Nc$ to the complex
		\[
			\B_0 \overset{f}{\lla} \B_1 \lla 0 \lla 0 \lla \dots
		\]
		concentrated in degrees $\{0,1\}$ yields an $\infty$-categorical version
		of Waldhausen's relative $S_{\bullet}$-construction. Besides its appearance
		in Waldhausen's own work, this construction also features prominently in
		\cite{dkss:schober} where it provides a local description of perverse schobers
		(cf. \cite{ks:schobers}).
	\item Let $\B$ be a stable $\infty$-category and let $\B[2]$ denote the chain complex
		\[
			0 \lla 0 \lla \B \lla 0 \lla \dots
		\]
		concentrated in degree $2$. Then $\Nc(\B[2])$ is an $\infty$-categorical version of
		Hesselholt-Madsen's $S^{2,1}_{\bullet}$-construction. As explained in
		\cite{hesselholt-madsen}, a duality on the category $\B$ furnishes $\Nc(\B[2])$
		with the structure of a real object which can then be utilized to upgrade the
		$K$-theory spectrum of $\B$ to a genuine $\ZZ/(2)$-equivariant spectrum. 

	\item\label{item:sk} Let $\B$ be a stable $\infty$-category and let $\B[k]$ denote the chain complex
		with $\B$ concentrated in degree $k$. Then $\Nc(\B[k])$ is an $\infty$-categorical version
		of the $k$-dimensional $S^{\langle k \rangle}_{\bullet}$-construction introduced for
		abelian categories in \cite{poguntke:higher}. These higher-dimensional Waldhausen
		constructions have an interesting interpretation in the context of higher algebraic
		$K$-theory: Let us denote by $\A^{\simeq}$ the Kan complex obtained from an
		$\infty$-category $\A$ by discarding noninvertible morphisms. Then, for
		every $k \ge 1$, there is a canonical weak equivalence of spaces
		\begin{equation}\label{eq:kdeloop}
				\Omega^k |\Nc(\B[k])^{\simeq}|  \simeq K(\B)
		\end{equation}
		which exhibits $|\Nc(\B[k])^{\simeq}|$ as a $k$-fold delooping of the $K$-theory
		space of $\B$. In fact, the sequence of spaces $\{|\Nc(\B[k])^{\simeq}|\}_{k \ge
		1}$ can be augmented to a spectrum which models the connective algebraic $K$-theory
		spectrum of $\B$. Interpreting the simplicial object $\Nc(\B([k]))$ as a categorification of
		the Eilenberg-MacLane space $\N(B[k])$, we observe that \eqref{eq:kdeloop} may be
		regarded as a categorification of the description of the Eilenberg-MacLane spectrum
		associated to an abelian group $B$ in terms of the sequence $\{\N(B[k])\}_{k \ge 1}$.
		A more detailed study of the $S^{\langle k \rangle}_{\bullet}$-constructions in the
		context of stable $\infty$-categories are the subject of an ongoing project with G.
		Jasso \cite{dj:sk} where we explore relations to higher Auslander-Reiten theory as
		introduced by Iyama \cite{iyama:higher}. 
\end{enumerate}

The $k$th cohomology group of a topological space $X$ with coefficients in an abelian group $B$ can
be described as homotopy classes of maps from $X$ to $K(B,k)$. The interpretation of the
$2$-simplicial stable $\infty$-category $\Nc(\B[k])$ as a categorified Eilenberg-MacLane space
predicts the existence of a categorified notion of cohomology. This circle of ideas will be explored in
future work with a view towards applications to topological Fukaya categories. For $k=1$, the
results of \cite{dk:triangulated, d:a1homotopy} can be interpreted as describing the topological
Fukaya category of a marked Riemann surface as categorified a relative $1$st cohomology group.\\

\noindent
{\bf Acknowledgements.} It is a pleasure to thank Rune Haugseng, Gustavo Jasso, Dima Kaledin,
Mikhail Kapranov, Jacob Lurie, Thomas Nikolaus, Thomas Poguntke, Vadim Schechtman, Ed Segal, and
Nicolo Sibilla for inspiring conversations that helped shape my perspective on the material
presented in this work.  Specifically, I would like to thank Mikhail Kapranov and Vadim Schechtman
for many discussions on perverse schobers, which are one of the sources of inspiration for this
work, and I am indebted to Thomas Nikolaus for conversations which convinced me to focus on
$2$-categorical methods.

\section{The classical Dold-Kan correspondence} \label{sec:dk}

Let $\Ab$ denote the category of abelian groups. The Dold-Kan correspondence establishes an
equivalence between the category $\Ab_{\Delta}$ of simplicial abelian groups and the category
$\Ch_{\ge 0}(\Ab)$ of connective chain complexes of abelian groups. We present a particular proof of
this result which is designed so that the proof of our main result, provided in \S
\ref{sec:cdk}, can be regarded as a step-by-step categorification of the involved arguments.

Let $A_{\bullet}$ be a simplicial abelian group. The associated chain complex $(A_{\bullet},d)$ with 
differential $d = \sum_{i = 0}^n (-1)^i d_i$ admits two subcomplexes $(\overline{A}_{\bullet},d)$
and $(D_{\bullet},d)$ where, for $n \ge 0$, we have
\[
	\overline{A}_n := \bigcap_{i = 1}^n \ker(d_i) \subset A_n,
\]
and $D_n \subset A_n$ is the subgroup generated by the degenerate $n$-simplices. 
For each element $\underline{j} = (j_1,\dots,j_n)$ of the cube $\{0,1\}^n$, we consider
\[
	f_{\underline{j}}:\; [n] \to [n], \; i \mapsto i - 1 + j_i,
\]
setting $j_0 = 1$, and introduce the map
\begin{equation}\label{eq:pi}
		\pi: A_n \to A_n, \; a \mapsto \sum_{\underline{j} \in \{0,1\}^n} (-1)^{n - |\underline{j}|}
		f_{\underline{j}}^* a
\end{equation}
where $|\underline{j}| = \sum_i j_i$.

\begin{prop}\label{prop:pi} Let $n \ge 0$, and let $D_n \subset A_n$ be the subgroup generated by the degenerate
	simplices. Then the map $\pi: A_n \to A_n$ is a retraction onto $\overline{A}_n$ with kernel
	$D_n$. In particular, it induces an isomorphism
	\[
		A_n / D_n \cong \overline{A}_n.
	\]
\end{prop}
\begin{proof}
	To show that $\im(\pi) \subset \overline{A}_n$, we observe that, for $0 < i \le n$, the face
	map $d_i$ maps the two faces of the cube $\{f_{\underline{j}}^*a\}$ that are orthogonal to the $i$th
	coordinate direction to the same $(n-1)$-dimensional cubes in $A_{n-1}$. Since the
	contributions of these faces in the formula for $\pi$ appear with opposite signs, we obtain,
	for every $a \in A_n$, $d_i \pi(a) = 0$.
	A similar argument shows that $D_n \subset \ker(\pi)$: for $a = s_i a'$, we have $\pi(a) =
	0$ since the opposing faces orthogonal to the $i$th coordinate direction of the cube 
	$\{f_{\underline{j}}^*a\}$ cancel in formula \eqref{eq:pi}.
	We show that $\pi$ is a retraction. Let $a \in \overline{A}_n$. For $\underline{j} \ne
	(1,\dots,1)$, the map $f_{\underline{j}}$ factors through some face map $\partial_i: [n-1] \to
	[n]$, $i >0$, so that $f_{\underline{j}}^*a = 0$. Since $f_{(1,\dots,1)} = \id$, we have
	$\pi(a) = a$ so that $\pi$ is a retraction.
	Finally, formula \eqref{eq:pi} implies that $A_n = \overline{A}_n + D_n$ which in
	combination with the statements established above yields $A_n = \overline{A}_n \oplus D_n$
	so that $D_n = \ker(\pi)$.
\end{proof}

\begin{cor} We have an isomorphism of complexes
	\[
		(A_{\bullet}, d) \cong  (\overline{A}_{\bullet}, d_0) \oplus (D_{\bullet},d) 
	\]
	where the projection onto the first summand is given by $\pi$.
\end{cor}
\begin{proof} Immediate from Proposition \ref{prop:pi}.
\end{proof}

We consider the functor
\[
	\C:\; \Ab_{\Delta} \lra \Ch_{\ge 0}(\Ab),\; A_{\bullet} \mapsto (\overline{A}_{\bullet},d_0)
\]
referring to $\C(A_{\bullet})$ as the {\em normalized chain complex} associated to $A_{\bullet}$.
The functor $\C$ is part of an adjunction
\[
	\C:\; \Ab_{\Delta} \llra \Ch_{\ge 0}(\Ab)\; : \N
\]
where the right adjoint $\N$ is, for formal reasons, given as follows: 
For a chain complex $B_{\bullet}$, we have 
\[
	\N(B_{\bullet})_n = \Hom_{\Ch_{\ge 0}(\Ab)}(\C(\ZZ \Delta^{n}), B_{\bullet})
\]
where $\C(\ZZ \Delta^{n})$ denotes the normalized chain complex of the $n$-simplex. More explicitly,
an $n$-simplex in $\N(B_{\bullet})$ is given by a collection $\{b_{\sigma}\}$
where $\sigma$ runs over all inclusions $\sigma: [m] \hra [n]$ and $b_{\sigma} \in B_m$ subject to the equations
\[
	d b_{\sigma} = \sum_{i = 0}^{m} (-1)^i b_{\sigma \circ \partial_i}.
\]

\begin{thm}\label{thm:cdk} The adjunction 
\[
	\C:\;\Ab_{\Delta} \llra \Ch_{\ge 0}(\Ab)\;:\N
\]
is a pair of inverse equivalences.
\end{thm}
\begin{proof}
	We analyze the counit of the adjunction $\C(\N(B_{\bullet})) \to B_{\bullet}$. An $n$-simplex
	in $\N(B_{\bullet})$ is given by a collection $\{b_{\sigma}\}$ and the counit maps this
	collection to $b_{\id} \in B_n$. The condition that $\{b_{\sigma}\}$ be a normalized chain
	translates to the requirement that $b_{\sigma} = 0$ for all $\sigma$ that factor through one
	of the face maps $\partial_i$, $i>0$. But this implies that the only possibly nonzero
	elements are $b_{\id}$ and $b_{\partial_0}$. Further, the element $b_{\partial_0}$ is
	determined as the image of $b_{\id}$ under $d$. Therefore, the counit is an isomorphism.

	The unit $u: A_{\bullet} \to \N(\C(A_{\bullet}))$ is given by associating to $a$ in $A_n$ the
	$n$-simplex in $\N((A_{\bullet}, d))$ given by the collection $\{a_{\sigma}\}$ with
	$a_{\sigma} = \sigma^*a$ and then postcomposing with the map
	\[
		\N(\pi): \N((A_{\bullet}, d)) \lra \N((\overline{A}_{\bullet}, d_0)).
	\]
	By an argument similar as for the counit, it is immediate that $\C(u)$ is an isomorphism.
	We conclude the proof in virtue of Proposition \ref{prop:conservative} below. 
\end{proof}

\begin{prop}\label{prop:conservative} The normalized chains functor $\C$ is conservative: a
	morphism $f: A_{\bullet} \to A'_{\bullet}$ of simplicial abelian groups is an isomorphism if
	and only if $\C(f)$ is an isomorphism of chain complexes.
\end{prop}
\begin{proof} Given a simplicial abelian group $A_{\bullet}$, we denote by $\P(A_{\bullet})$ the
	simplicial abelian group obtained by pullback of $A_{\bullet}$ along the functor
	\[
		\Delta \to \Delta,\; [n] \mapsto [n] \ast [0]
	\]
	so that $\P(A_{\bullet})_n = A_{n+1}$. The collection of omitted face maps
	\[
		d_n : A_{n} \to A_{n-1}
	\]
	defines a map of simplicial abelian groups $d: \P(A_{\bullet}) \to A_{\bullet}$. We denote the
	kernel of $d$ by $\Omega(A_{\bullet})$. Similarly, we define for a connective chain complex
	$B_{\bullet}$ the chain complex $\Omega(B_{\bullet})$ as the shift
	\[
		B_1 \leftarrow B_2 \leftarrow B_3 \leftarrow \dots, 
	\]
	ommitting $B_0$.
	It is immediate from the definitions that we have a natural isomorphism
	\begin{equation}\label{eq:omega}
			\C(\Omega(A_{\bullet})) \cong \Omega\C(A_{\bullet}).
	\end{equation}
	We show by induction on $n \ge 0$ that, for every map $f: A_{\bullet} \to A'_{\bullet}$ such
	that $C(f)$ is an isomorphism of chain complexes, the map $f_n: A_n \to A'_n$ is an
	isomorphism. 
	For $n = 0$, the claim is apparent. Assume the induction hypothesis holds for
	a fixed $n \ge 0$ and all maps of simplicial abelian groups. For a given map $f: A_{\bullet}
	\to A'_{\bullet}$, consider the diagram simplicial abelian groups
	\[
		\xymatrix{
			\Omega(A_{\bullet}) \ar[r] \ar[d]^{\overline{f}}& \P(A_{\bullet})
			\ar[d]^{\P(f)} \ar[r] &
			A_{\bullet} \ar[d]^f\\
			\Omega(A'_{\bullet}) \ar[r]& \P(A'_{\bullet}) \ar[r] &
			A'_{\bullet} }
	\]
	where the horizontal sequences are short exact. 
	Then the induction hypothesis implies that $\overline{f}_n$ and $f_n$ are isomorphisms so
	that $\P(f)_n = f_{n+1}$ is an isomorphism as well.
\end{proof}

\section{The categorified Dold-Kan correspondence}\label{sec:cdk}
 
In this section, we prove the main result of this work: a categorification of the Dold-Kan
correspondence relating simplicial objects and connective chain complexes with values in the
category of stable $\infty$-categories. We begin by defining these notions in detail.

\subsection{Basic definitions}
\label{sec:basic}

\subsubsection{Model for $(\infty,2)$-categories}\label{subsec:2cat} 

Let $\sSet$ denote the category of simplicial sets. Following \cite{lurie:htt}, we define an {\em
$\infty$-category} to be a simplicial set satisfying the inner horn filling conditions. We define
$\Cat_{\infty}$ to be the full subcategory of $\sSet$ spanned by the $\infty$-categories. 
$\infty$-categories are the fibrant objects of a model structure on the category of marked
simplicial sets $\msSet$ with marked edges given by the equivalences. As explained in
\cite{lurie:2cat}, the category of $\msSet$-enriched categories carries a model structure which can
be regarded as a model for the theory of $(\infty,2)$-categories. The $(\infty,2)$-categorical
structures that appear in this work will be organized within this model and related to
other models via the theory developed in \cite{lurie:2cat}. The fibrant objects within this model
structure can be identified with the $\Cat_{\infty}$-enriched categories.  

\subsubsection{Stable $\infty$-categories}\label{subsec:stable}

The simplicial set of functors between a pair of $\infty$-categories forms another $\infty$-category
so that $\Cat_{\infty}$ becomes a $\Cat_{\infty}$-enriched category with respect to the Cartesian
monoidal structure. Marking equivalences in the various functor $\infty$-categories, $\Cat_{\infty}$
becomes a fibrant object in the model category of $\msSet$-enriched categories from
\ref{subsec:2cat}. We may therefore interpret $\Cat_{\infty}$ as a specific model for the
$(\infty,2)$-category of $\infty$-categories. We further denote by $\St \subset \Cat_{\infty}$ the
$\Cat_{\infty}$-enriched subcategory with {\em stable} $\infty$-categories as objects and functor
$\infty$-categories spanned by the {\em exact} functors in the sense of \cite{lurie:ha}.

\subsubsection{The simplex $2$-category}\label{subsec:simplex}

By a $2$-category, we mean a category enriched in categories. A $2$-category defines a
$\msSet$-enriched category by passing to nerves of the enriched mapping categories and marking
equivalences. We will typically leave this passage implicit so that, referring to \ref{subsec:2cat},
we may consider any $2$-category as an $(\infty,2)$-category. We denote by $\CCat$ the $2$-category
of small categories and by $\DDelta \subset \CCat$ the full $2$-subcategory spanned by the standard
ordinals $\{[n]\}$, considered as categories. 

\subsubsection{Localization}\label{subsec:local}

We will construct $\infty$-categories via localization: Given a small category $\Cc$ and a
set of morphisms $W$, there is an associated $\infty$-category $\L_{W}\Cc$, equipped with a functor
$\N(\Cc) \to \L_{W}\Cc$ universal among all functors that send $W$ to equivalences (cf.
\cite[1.3.4]{lurie:ha}). 

\subsubsection{$2$-simplicial stable $\infty$-categories}\label{subsec:simpstab}

We denote by $\St_{\DDelta}$ the category of $\Cat_{\infty}$-enriched functors from the opposite
$2$-simplex category $\DDelta^{(\op,-)}$ to $\St$. This category comes equipped with a collection of
weak equivalences given by those $\Cat_{\infty}$-enriched natural transformations that are levelwise
equivalences of stable $\infty$-categories. Via localization, we obtain a corresponding
$\infty$-category $\L(\St_{\DDelta})$ of $2$-simplicial stable $\infty$-categories.

\begin{rem} Let $\A_{\bullet} \in \St_{\DDelta}$, and let $n \ge 1$. The functor of $1$-categories
	$\Delta^{\op} \to \St_{\infty}$ underlying $\A_{\bullet}$ provides us with exact functors
	\[
		\begin{tikzcd}
			\A_{n-1} \arrow[shift left=1.5em]{rr}[description]{s_0}\arrow[shift
			right=1.5em]{rr}[description]{s_{n-1}} & \vdots &  \arrow[shift
			right=2.5em]{ll}[description]{d_0} 
			\arrow[shift left=2.5em]{ll}[description]{d_n} 
			\A_n.
		\end{tikzcd}
	\]
	The additional data captured by the $2$-functoriality of $\A_{\bullet}$ contains unit and
	counit transformations which exhibit a sequence of adjunctions
	\[
		d_0 \vdash s_0 \vdash d_1 \vdash \dots \vdash s_{n-1} \vdash d_n.
	\]
\end{rem}

\subsubsection{Connective chain complexes of stable $\infty$-categories}\label{subsec:chainstab}

We denote by $\NN$ the poset of nonnegative integers, considered as a category. We denote by
$\Fun(\NN^{\op},\St)$ the category of (strict) functors from the opposite category of $\NN$
to the category $\St$ of stable $\infty$-categories. We introduce the full subcategory
\[
	\Ch_{\ge 0}(\St) \subset \Fun(\NN^{\op},\St)
\]
given by those diagrams
\[
	\B_0 \overset{d}{\lla} \B_1 \overset{d}{\lla} \B_2 \overset{d}{\lla} \; \cdots
\]
of stable $\infty$-categories that satisfy the following condition:
for every $i \ge 0$, the functor $d^2: \B_{i+2} \to \B_i$ is a zero object of the
$\infty$-category $\Fun(\B_{i+2},\B_i)$.
The category $\Ch_{\ge 0}(\St)$ comes equipped with a class of weak equivalences given by those
natural transformations that are levelwise equivalences. We refer to the corresponding localization
$\L\Ch_{\ge 0}(\St)$ as the $\infty$-category of {\em connective chain complexes of stable
$\infty$-categories.} 

\begin{rem} At first sight, our definition of a complex of stable $\infty$-categories may seem too
	naive. For example, the analogous definition of a connective complex of {\em objects} in a stable
	$\infty$-category $\A$ as a functor
	\[
		X: \N(\NN^{\op}) \lra \A
	\]
	satisfying the condition $d^2 \simeq 0$ really {\em is} too naive. The reason is that,
	for every $i \ge 0$, there is a potentially nontrivial space of paths in the Kan complex
	$\Map_{\A}(X_{i+2},X_i)$ from $d^2$ to $0$. Following general principles, the condition $d^2 \simeq
	0$ should be replaced by the {\em choice} of a path in $\Map_{\A}(X_{i+2},X_i)$ between
	$d^2$ and $0$. Further, these choices are supposed to be part of a coherent system of
	trivializations $d^n \simeq 0$, $n \ge 2$, corresponding to trivializations of higher
	Massey products. This coherence data is important: it is, for example, needed to form the
	totalization of a complex. One way to codify all this data is to remember the complex in
	terms of the filtered object formed by the totalizations of its various truncations. This is
	the point of view taken in \cite{lurie:ha} where a connective complex of objects in a stable
	$\infty$-category corresponds to a functor
	\begin{equation}\label{eq:filtered}
			Y: \N(\NN) \lra \A
	\end{equation}
	without any further conditions. The actual terms of the complex captured by such a datum are
	then given as shifts of the cofibers of the maps $Y_i \to Y_{i+1}$. One concrete justification for
	this being a reasonable notion of a complex is provided by a Dold-Kan correspondence
	relating simplicial objects and connective chain complexes with values in a given stable
	$\infty$-category (\cite[1.2.4]{lurie:ha}).

	In contrast, given a chain complex of stable $\infty$-categories in our sense, the space of
	identifications $d^2 \simeq 0$ is contractible, since it is the space of equivalences between
	$d^2$ and the zero object $0$ in the $\infty$-category $\Fun(\B_{i+2},\B_i)$. Therefore, in
	this context, there is no analog of the coherent system of trivializations captured by
	$\eqref{eq:filtered}$.
\end{rem}

\subsection{Statement of the theorem}

Using the terminology introduced in \S \ref{sec:basic}, we may formulate the main theorem.

\begin{thm}\label{thm:catdk} There exist functors
	\begin{equation}
		\Cc: \St_{\DDelta} \lra \Ch_{\ge 0}(\St)
	\end{equation}
	and 
	\begin{equation}
		\Nc: \Ch_{\ge 0}(\St) \lra \St_{\DDelta}
	\end{equation}
	which induce a pair of inverse equivalences
	\[
		\Cc:  \L\St_{\DDelta} \overset{\simeq}{\llra} \L\Ch_{\ge 0}(\St): \Nc
	\]
	of $\infty$-categories.
\end{thm}

\begin{rem}
	The classical Dold-Kan correspondence generalizes to categories that are additive and idempotent
	complete. While the $(\infty,2)$-category $\St$ does not have direct analogs of these two
	properties, it {\em does} admit certain categorified variants: 
	\begin{enumerate}[label=(\roman{*})]
		\item Given two functors $f$ and $g$ between stable $\infty$-categories $\A$ and $\A'$,
			equipped with a natural transformation $\eta: f \Rightarrow g$, we may form the
			cone of $\eta$ as a replacement for the difference of two maps. 
		\item Given an idempotent $e: \A \to \A$ that arises from a fully faithful embedding $i: \A' \subset
			\A$ with right adjoint $q: \A \to \A'$ as $e = i \circ q$, then there is a canonical
			natural transformation $e \Rightarrow \id_{\A}$ whose cone will be a projector
			onto the subcategory $\ker(q) \subset \A$. Together, the subcategories
			$\ker(q)$ and $\A'$ form a {\em semiorthogonal decomposition} of $\A$ (cf.
			\cite{bk:semi}). 
	\end{enumerate}
	Our proof of Theorem \ref{thm:catdk} relies on a systematic utilization of these features
	of $\St$.  Abelian categories form a pleasant class of categories for which the Dold-Kan
	correspondence holds. It is an interesting task to introduce a suitable categorified
	axiomatic framework of ``$2$-abelian'' $(\infty,2)$-categories. One basic requirement would
	be that the proof of the categorified Dold-Kan correspondence generalizes to this context.
\end{rem}

Before we construct the functors in Theorem \ref{thm:catdk} and provide its proof, we need some
preliminaries on Grothendieck constructions.

\subsection{Grothendieck constructions}

Let $C$ be a category and let $\Grp$ denote the category of small categories. Recall that, given a functor
\[
	F: C^{\op} \lra \Grp,
\]
the classical {\em Grothendieck construction} of $F$ is the category $\chi(F)$ with 
\begin{itemize}
	\item objects are pairs $(c,x)$ where $c \in C$ and $x \in F(c)$,
	\item a morphism from $(c,x)$ to $(c',x')$ consists of a morphism $f: c \to c'$ in $C$
		together with a morphism $x \to F(f)(x')$ in $F(c)$. 
\end{itemize}
The construction provides a functor
\[
	\chi: \Fun(C^{\op},\Grp) \lra \Fib_{/C}
\]
into the category $\Fib_{/C}$ of categories fibered in groupoids over $C$. The functor $\chi$ has a
left adjoint $\Xi$ given by 
\[
	\Xi(D \to C)(c) = D \times_{C} C_{c/}
\]
and a right adjoint $\Gamma$ provided by 
\[
	\Gamma(D \to C)(c) = \Fun_{C}(C_{/c},D).
\]
Upon localizing the categories $\Fun(C^{\op},\Grp)$ and $\Fib_{/C}$ along suitable classes of weak
equivalences, the functor $\chi$ becomes an equivalence with inverse provided by both $\Xi$ and
$\Gamma$. As a central tool in his approach to higher category theory, Lurie has introduced various
generalizations of the adjunction
\[
	\Xi: \Fib_{/C} \llra \Fun(C^{\op},\Grp): \chi
\]
in the context of Quillen's theory of model categories. In this work, we will use various explicit versions
of these Grothendieck constructions and a description of their inverses in terms of generalizations
of the right adjoint functor $\Gamma$. In this section, we survey the theory in the form it will be
applied below.

\subsubsection{The $(\infty,1)$-categorical Grothendieck construction}
\label{subsec:grothendieck}

Let $C$ be an ordinary category and let 
\[
	F: C^{\op} \lra \Cat_{\infty}
\]
be a functor into the category $\Cat_{\infty}$ of $\infty$-categories, considered as a full
subcategory of the category $\Set_{\Delta}$ of simplicial sets. 

\begin{defi} We define a simplicial set $\chi(F)$
as follows: An $n$-simplex of $\chi(F)$ consists of 
\begin{enumerate}
	\item an $n$-simplex $\sigma: [n] \to C$ of the nerve of $C$,
	\item for every $I \subset [n]$, a functor
		\[
			\Delta^I \lra F(\sigma(\min(I)))
		\]
		such that, for every $I \subset J \subset [n]$, the diagram
		\[
		\xymatrix{
			\Delta^I \ar[r]\ar[d] & F(\sigma(\min(I))\ar[d]\\
			\Delta^J \ar[r] & F(\sigma(\min(J)).
		}
		\]
\end{enumerate}
The simplicial set $\chi(F)$ comes equipped with an apparent forgetful map $\chi(F) \to \N(C)$. We
consider $\chi(F)$ as an object of $\msSet_{/\N(C)}$ marking the Cartesian edges, and refer to it
as the {\em Grothendieck construction of $F$}. 
\end{defi}

As a consequence of results in \cite{lurie:htt}, the
Grothendieck construction induces an equivalence of $\infty$-categories
\[
	\chi: \L\Fun(C^{\op}, \Cat_{\infty}) \overset{\simeq}{\lra} \L(\msSet_{/\N(C)})^{\circ}
\]
obtained by localizing along levelwise and fiberwise categorical equivalences. The right-hand
symbol $\circ$ signifies that we restrict to the full subcategory of Cartesian fibrations with
Cartesian edges marked.

We will use an explicit model for an inverse equivalence to $\chi$ which we now
construct. For an object $c \in C$, we denote by $C_{/c}$ the overcategory of $c$. Keeping track
of the forgetful functor $C_{/c} \to C$, the various overcategories organize to define a functor
\[
	C \lra \msSet_{/\N(C)},\; c \mapsto \N(C_{/c})^{\#}.
\]
Given an object $X \in \msSet_{/\N(C)}$, we thus obtain a functor
\[
	\Gamma(X): C^{\op} \lra \Cat_{\infty},\; c \mapsto \Map^{\flat}_{\N(C)}(\N(C_{/c})^{\#}, X).
\]

\begin{prop}\label{prop:groth} There is a natural transformation
	\begin{equation}
			\eta: \id \lra \Gamma \circ \chi
	\end{equation}
	which is a levelwise weak equivalence. In particular, upon localization, the functor
	$\Gamma$ defines an inverse to $\chi$.
\end{prop}
\begin{proof}
	For a functor $F: C^{\op} \lra \Cat_{\infty}$ and $c \in C$, we provide a map of simplicial
	sets 
	\begin{equation}\label{eq:groth}
		\eta_F(c): F(c) \lra \Gamma(\chi(F))(c)
	\end{equation}
	which, by adjunction, can be identified with a map
	\[
			F(c) \times \N(C_{/c}) \lra \chi(F)
	\]
	of simplicial sets over $\N(C)$. We define this map by associating to an $n$-simplex
	$(x, c_0 \to c_1 \to \dots \to c_n \to c)$ of $F(c) \times \N(C_{/c})$ the $n$-simplex of
	$\chi(F)$ that is given by
	\begin{enumerate}
		\item the $n$-simplex $\sigma = c_0 \to \dots \to c_n$ in $\N(C)$ 
		\item for $I \subset [n]$, the $I$-simplex of $F(\sigma(\min(I)))$ obtained as the composite
			\[
				\Delta^I \to \Delta^n \overset{x}{\to} F(c) \to F(\sigma(\min(I))). 
			\]
	\end{enumerate}
	This association is natural in $c$ and provides the value of the transformation $\eta$ at
	$F$. 

	To show that $\eta$ is a weak equivalence, we need to show that, for every $F: C^{\op} \to
	\Cat_{\infty}$ and for every $c \in C$, the map $\eta_F(c)$ from \eqref{eq:groth} is an
	equivalence of $\infty$-categories. To this end, we note that there is a natural evaluation map
	\[
		\on{ev}_c: \Map^{\flat}_{\N(C)}(\N(C_{/c})^{\#}, \chi(F)) \to F(c)
	\]
	obtained by pullback along the map $\Delta^0 \to \N(C_{/c})$ corresponding to the
	vertex $\id: c \to c$. The composite $\on{ev}_c \circ \eta_F(c)$ equals the identity. The map
	$\on{ev}_c$ is an equivalence by \cite[4.3.2.15]{lurie:htt}, since the requirement to map all
	edges in $\N(C_{/c})$ to Cartesian edges is equivalent to being a relative right Kan
	extension along $\Delta^0 \to \N(C_{/c})$. We conclude in virtue of the two-out-of-three
	property of equivalences of $\infty$-categories.
\end{proof}

\subsubsection{$2$-categorical terminology}
\label{sec:2cat}

Let $\CC$ be a $2$-category. We denote the category of morphisms between objects $x$ and $y$ by
$\CC(x,y)$. We denote by $\CC^{(\op,-)}$ the $2$-category with 
\[
	\CC^{(\op,-)}(c,c') = \CC(c',c)
\]
and by $\CC^{(-,\op)}$ the $2$-category with 
\[
	\CC^{(-,\op)}(c,c') = \CC(c,c')^{\op}
\]
and further $\CC^{(\op,\op)} = (\CC^{(\op,-)})^{(-,\op)}$.

\begin{exa} The $2$-category $\CCat$ of small categories admits a self-duality 
	\[
		\CCat \overset{\simeq}{\lra} \CCat^{(-,\op)}, \C \mapsto \C^{\op}.
	\]
\end{exa}

We introduce $2$-categorical versions of undercategories and overcategories. Since there is
potential confusion with the orientation of the $2$-morphisms, we provide explicit descriptions of
both. For an object $c$ in a $2$-category $\CC$, we define the {\em lax undercategory} $\CC_{c/}$
with 
\begin{itemize}
	\item objects given by $1$-morphisms $c \to x$,
	\item a $1$-morphism from $\varphi:c \to x$ to $\psi:c \to y$ consists of a
		$2$-commutative triangle 
		
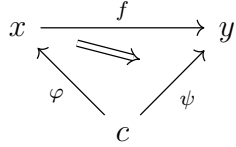
\[
		\begin{tikzcd}[sep=2em]
			|[alias=X]|x \arrow{rr}{f} & & {y}\\
			&{c}\arrow{ur}[""{name=foo},swap]{\psi} \arrow{ul}{\varphi}  & 
			\arrow[Rightarrow, from=X, to=foo, swap, shorten <=.5cm,shorten >=.5cm]
		\end{tikzcd}
		\]
		where $f$ is a $1$-morphism in $\CC$, 
	\item a $2$-morphism $\gamma$ from $f$ to $g$ is given by a $2$-commutative diagram
		
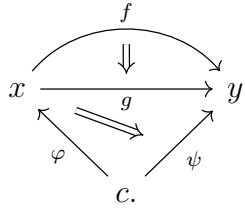
\[
		\begin{tikzcd}[sep=2em]
			|[alias=X]|x \arrow[""{name=bar},swap]{rr}{g} \arrow[""{name=gbar},bend left=50]{rr}{f}  & & {y}
			\arrow[Rightarrow, from=gbar, to=bar,swap, shorten >=.2cm, shorten
			<=.2cm]\\
			&{c.}\arrow[""{name=foo}, swap]{ur}{\psi} \arrow{ul}{\varphi} & 
			\arrow[Rightarrow, from=X, to=foo, swap, shorten <=.5cm,shorten >=.5cm]
		\end{tikzcd}
		\]
\end{itemize}

For every object $c$ in $\CC$, we define the {\em lax overcategory} $\CC_{/c}$ with 
\begin{itemize}
	\item objects given by $1$-morphisms $x \to c$,
	\item a $1$-morphism from $\varphi:x \to c$ to $\psi:y \to c$ consists of a
		$2$-commutative triangle 
		
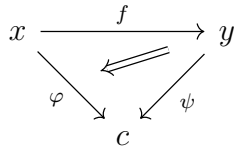
\[
		\begin{tikzcd}[sep=2em]
			{x} \arrow{rr}{f} \arrow[""{name=foo}]{dr}[swap]{\varphi} & & |[alias=Y]|y
			\arrow{dl}{\psi}\\
			& {c} & 
			\arrow[Rightarrow, from=Y, to=foo, swap, shorten <=.5cm,shorten >=.2cm]
		\end{tikzcd}
		\]
		where $f$ is a $1$-morphism in $\CC$, 
	\item a $2$-morphism $\gamma$ from $f$ to $g$ is given by a $2$-commutative diagram
		
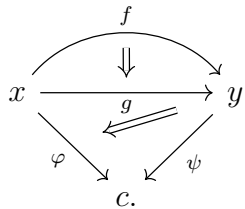
\[
		\begin{tikzcd}[sep=2em]
			{x} \arrow[""{name=bar},swap]{rr}{g} \arrow[""{name=gbar},bend left=50]{rr}{f} \arrow[""{name=foo}]{dr}[swap]{\varphi} & & |[alias=Y]|y
			\arrow{dl}{\psi}
			\arrow[Rightarrow, from=gbar, to=bar, shorten >=.2cm, shorten <=.2cm]\\
			& {c.} & 
			\arrow[Rightarrow, from=Y, to=foo, swap, shorten <=.5cm,shorten >=.2cm]
		\end{tikzcd}
		\]
\end{itemize}
	
\subsubsection{The $(\infty,2)$-categorical Grothendieck construction}
\label{subsec:2groth}

In this section, we introduce a combinatorial variant of the unstraightening functor for
$(\infty,2)$-categories introduced in \cite{lurie:2cat}. It can be applied to any strict $2$-functor
\[
	\CC \lra \Cat_{\infty}
\]
where $\CC$ is a $2$-category. It is a lax analog of the relative nerve construction, provided in
\cite{lurie:htt} as a combinatorial alternative to the straightening construction applicable to strict
functors $\C \to \Cat_{\infty}$ where $\C$ is an ordinary category. 
	
For every nonempty finite linearly ordered set $I$, we define a $2$-category $\SSigma^I$ as follows:
\begin{itemize}
	\item The set of objects of $\SSigma^I$ is $I$,
	\item The category $\SSigma^I(i,j)$ of morphisms between objects $i$ and $j$ is
		the poset consisting of those subsets $S \subset I$ satisfying $\min S = i$ and $\max S = j$. 
	\item The composition law is given by the formula
		\[
			\SSigma^I(i,j) \times \SSigma^I(j,k) \lra \SSigma^I(i,k),\; (S, S') \mapsto S \cup S'.
		\]
\end{itemize}
The various $2$-categories $\SSigma^{[n]}$, $n \ge 0$, assemble to a functor
\begin{equation}\label{eq:u2}
		\Delta \lra \Cat_2
\end{equation}
into the category $\Cat_2$ of $2$-categories. 

\begin{defi}
Let $\CC$ be a $2$-category. We define the nerve $\Nsc(\CC)$ of $\CC$ to be the scaled simplicial
set with
\[
	\Nsc(\CC)_n = \Fun(\SSigma^{[n]}, \CC)
\]
with functoriality in $n$ provided by \eqref{eq:u2}. The thin $2$-simplices are the ones that
correspond to invertible natural transformations.
\end{defi}

For $n \ge 0$, and $\emptyset \neq I \subset [n]$, we denote by $G(I)$, the $1$-category obtained
from the $2$-category
\[
	((\SSigma^I)^{(-,\op)})_{\min(I)/} 
\]
by discarding all $2$-morphisms. For $\emptyset \neq I \subset J$, we have a pullback functor 
\begin{equation}\label{eq:2simplex}
		\SSigma^J(\min(I),\min(J))^{\op} \times G(J) \lra G(I).
\end{equation}

\begin{defi}
Let 
\[
	F: \CC^{(\op,\op)} \lra \Cat_{\infty}
\]
be a $\Cat_{\infty}$-enriched functor where we interpret $\CC^{(\op, \op)}$ as $\Cat_{\infty}$-enriched via
passage to nerves of the morphism categories. We introduce a simplicial set $\CChi(F)$, called the
{\em lax Grothendieck construction} of $F$ as follows: An $n$-simplex in $\CChi(F)$ consists of
\begin{enumerate}
	\item a functor $\sigma: \SSigma^n \to \CC$ of $2$-categories, 
	\item for every $\emptyset \ne I \subset [n]$, a functor $\N(G(I)) \to F(\sigma(\min(I)))$ so
		that, for every $ \emptyset \ne I
		\subset J \subset [n]$, the diagram
		\[
		\begin{tikzcd}
			\N(\SSigma^J(\min(I),\min(J))^{\op}) \times \N(G(J))
			\arrow{r}\arrow{d} & \N(G(I))\arrow{d}\\
			\N(\CC(\sigma(\min(I)),\sigma(\min(J)))^{\op}) \times F(\sigma(\min(J))) \arrow{r} &
			F(\sigma(\min(I)))
		\end{tikzcd}
		\]
		commutes where the top row is obtained from \eqref{eq:2simplex} by passing to
		nerves.
\end{enumerate}
By construction, the lax Grothendieck construction comes equipped with a forgetful
functor $\pi: \CChi(F) \to \Nsc(\CC)$. We further introduce a marking on $\CChi(F)$ consisting of
the $\pi$-Cartesian edges so that we have $\CChi(F) \in \msSet_{/\Nsc(\CC)}$.
\end{defi}

We summarize some basic properties whose, in parts somewhat technical, proofs are deferred to
\cite{d:grothendieck}.

\begin{prop}\label{prop:laxbasic} Let $\CC$ be a $2$-category, and let
	\[
		F: \CC^{(\op,\op)} \lra \Cat_{\infty}
	\]
	be a $\Cat_{\infty}$-enriched functor. Then
	\begin{enumerate}
		\item The lax Grothendieck construction $\pi: \CChi(F) \to \Nsc(\CC)$ is a locally
			Cartesian fibration which is Cartesian over every thin $2$-simplex of
			$\Nsc(\CC)$. 
		\item For every $\Cat$-enriched functor $\DD \to \CC$, we have
			\[
				\CChi(F) \times_{\Nsc(\CC)} \Nsc(\DD) \cong \CChi(F|\DD).
			\]
		\item \label{prop:laxbasic:3} Suppose that $\CC$ is a $2$-category with discrete morphism categories which
			we may therefore identify with a $1$-category. Then restriction along the
			functor 
			\[
				G(I) \to I, (\min(I) \to j) \mapsto j
			\]
			induces a map
			\begin{equation}\label{eq:ordlax}
					\chi(F) \lra \CChi(F)
			\end{equation}
			between the ordinary and lax Grothendieck constructions which is a fiberwise
			equivalence of Cartesian fibrations.
	\end{enumerate}
\end{prop}

Let $\CC$ be a $2$-category and $c$ an object of $\CC$. Consider the lax overcategory $\CC_{/c}$ as defined in
\S \ref{sec:2cat}. The forgetful functor $\CC_{/c} \to \CC$ induces a map $\Nsc(\CC_{/c}) \to
\Nsc(\CC)$ of simplicial sets. We further introduce a marking on $\Nsc(\CC_{/c})$ given by those
edges where the corresponding $2$-morphism is an isomorphism. Thus, we have $\Nsc(\CC_{/c}) \in
\msSet_{/\Nsc(\CC)}$. The functoriality of this construction in $c$ is captured by a
$\Cat_{\infty}$-enriched functor
\begin{equation}\label{eq:funct}
	\CC^{(-,\op)} \lra \msSet_{/\Nsc(\CC)},\; c \mapsto \Nsc(\CC_{/c})
\end{equation}

We now define for $X \in \msSet_{/\Nsc(\CC)}$ and $c \in \CC$, the marked simplicial set
\[
	\GGamma(X)(c) = \Map^{\#}_{\Nsc(\CC)}(\N(\CC_{/c}), X).
\]
Pulling back \eqref{eq:funct}, this construction is functorial in $c$ and defines a functor
\[
	\GGamma(X): \CC^{(\op,\op)} \lra \msSet.
\]
The additional functoriality in $X$ provides
\[
	\GGamma: \msSet_{/\Nsc(\CC)} \lra \Fun_{\msSet}(\CC^{\op}, \msSet),\; X \mapsto \GGamma(X).
\]

\begin{prop}\label{prop:2groth} Let $F: \CC^{(\op,\op)} \to \Cat_{\infty}$ be a
	$\Cat_{\infty}$-enriched functor. Then there is a natural weak equivalence
		\[
			\eta: \id \overset{\simeq}{\lra} \GGamma \circ \CChi
		\]
		of endofunctors of $\Fun(\CC^{(\op,\op)}, \Cat_{\infty})$.
\end{prop}
\begin{proof} 
	The argument is similar to the proof of Proposition \ref{prop:groth}: For every $F:
	\CC^{(\op,\op)} \to \Cat_{\infty}$, there is an explicit natural map
	\[
		\eta_F: F \lra \GGamma \circ \CChi(F)
	\]
	constructed as follows: For every $c \in \CC$, we have to provide a map of simplicial sets
	\[
		F(c) \times \Nsc(\CC_{/c}) \lra \CChi(F)
	\]
	over $\Nsc(\CC)$. Let $(\sigma,\tau)$ be an $n$-simplex of $F(c) \times \Nsc(\CC_{/c})$, and
	denote by $\alpha$ the $n$-simplex in $\Nsc(\CC)$ obtained by postcomposing $\tau$ with the
	map $\Nsc(\CC_{/c}) \to \Nsc(\CC)$. Note that for $I \subset [n]$, we have an induced functor of categories
	\[
		(\SSigma^I)^{(-,\op)}_{\min(I)/} \lra \CC(\alpha(\min(I)), c)^{\op}, (\min(I) \to j) \mapsto
		\alpha(\min(I)) \to c
	\]
	given by postcomposition with the given map $\alpha(j) \to c$. Further, the
	$\Cat_{\infty}$-enriched functor $F$ provides a map 
	\[
		\N(\CC(\alpha(\min(I)), c)^{\op}) \times F(c) \lra F(\alpha(\min(I)))
	\]
	evaluation at $\sigma$ then provides a map
	\[
		\N((\SSigma^I)^{(-,\op)}_{\min(I)/}) \lra F(\alpha(\min(I)))
	\]
	via pullback of the image of $\sigma$ in $F(\alpha(\min(I)))$ along the map $[k]
	\to I \subset [n]$ corresponding to a $k$-simplex in $\N((\SSigma^I)^{(-,\op)}_{\min(I)/})$.
	The collection of these maps for the various nonempty subsets $I \subset [n]$ defines the
	desired $n$-simplex in $\CChi(F)$.

	To see that the resulting map $F \lra \GGamma \circ \CChi$ is a weak equivalence, we argue
	as follows. For every object $c \in \CC$, we consider the evalutation map 
	\[
		\on{ev}(c): \GGamma \circ \CChi(F)(c) = \Map^{\#}_{\Nsc(\CC)}(\N(\CC_{/c}), \CChi(F)) \lra
		\CChi(F)_{c}
	\]
	given by pullback along $\{c \overset{\id}{\to} c\} \to \N(\CC_{/c})$. By the argument of
	\cite[4.1.8]{lurie:2cat}, the opposite of this map is ${\mathfrak
	P}_{\Nsc(\CC)^{\op}}$-anodyne, so that $\on{ev}(c)$ is an equivalence of
	$\infty$-categories. Further, the composite $\on{ev}(c) \circ \eta_F(c)$ as the map
	$\chi(F)_c \to \CChi(F)_c$ induced by \eqref{eq:ordlax} on fibers over $c$. Hence, by
	Proposition \ref{prop:laxbasic}\ref{prop:laxbasic:3}, it is an equivalence of $\infty$-categories. We conclude
	by two-out-of-three.
\end{proof}

\subsection{The categorified normalized chains functor $\Cc$}

We provide the definition of the functor $\Cc$ from Theorem \ref{thm:catdk} whose construction is
a mutatis mutandis modification of the normalized chains functor appearing in the classical Dold-Kan
correspondence.

\begin{defi}
Given a $2$-simplicial stable $\infty$-category $\A_{\bullet}$, we define, for $n \ge 0$, the stable
$\infty$-category
\[
	\overline{\A}_n \subset \A_n
\]
as the full subcategory spanned by those vertices $a$ of $\A_n$ such that, for every $1 \le i \le n$,
the object $d_i(a)$ is a zero object in $\A_{n-1}$. As a result of the relations among the face maps in
$\Delta$, the various functors $d_0: \A_n \to \A_{n-1}$, $n > 0$, restrict to define a chain
complex
\[
	\overline{\A}_0 \overset{d_0}{\lla} \overline{\A}_1 \overset{d_0}{\lla} \overline{\A}_2
	\overset{d_0}{\lla} \; \cdots
\]
of stable $\infty$-categories. We denote this chain complex by $\Cc(\A_{\bullet})$ and refer to it
as the {\em categorified normalized chain complex} associated to $\A_{\bullet}$. 
\end{defi}

The construction $\A_{\bullet} \mapsto \Cc(\A_{\bullet})$ on objects extends to define a functor
\[
	\Cc: \St_{\DDelta} \lra \Ch_{\ge 0}(\St),\; \A_{\bullet} \mapsto (\overline{\A}_{\bullet},
	d_0)
\]
which preserves weak equivalences.

\subsection{The categorified Dold-Kan nerve $\Nc$}

We provide a definition of the functor $\Nc$ from Theorem \ref{thm:catdk}. It can be regarded as a
categorification of the classical Dold-Kan nerve appearing in \S \ref{sec:dk}.

Let $n \ge 0$. We denote by $\NN_{/[n]}$ the following lax version of the comma category of $[n] \in
\DDelta$ with respect to the embedding $\NN \to \DDelta^{(-,\op)}$:
\begin{itemize}
	\item The objects of $\NN_{/[n]}$ are given by morphisms $\varphi: [m] \to [n]$ in
		$\DDelta$. 
	\item A morphism from $\varphi:[m] \to [n]$ to $\varphi':[m'] \to [n]$ consists of a
		$2$-commutative triangle 
		\[
		\begin{tikzcd}[sep=2em]
			{[m]} \arrow{rr}{f} \arrow[""{name=foo}]{dr}[swap]{\varphi} & & {[m']}
			\arrow{dl}{\varphi'}
			\arrow[Rightarrow, from=foo, swap, near start, "\eta", shorten >=.7cm, shorten <=.2cm]\\
			& {[n]} & 
		\end{tikzcd}
		\]
		where $f$ is the image of a morphism in $\NN$, i.e., an iteration of face maps
		$\partial_0$.
\end{itemize}
We equip the category $\NN_{/[n]}$ with the forgetful functor to $\NN$ 
and thus obtain a $\Cat$-enriched functor
\begin{equation}\label{eq:2over}
	\DDelta^{(-,\op)} \lra \Cat_{/\NN},\; [n] \mapsto \NN_{/[n]}.
\end{equation}

\begin{rem} The category $\NN_{/[n]}$ is in fact a poset. 
\end{rem}

\begin{defi}\label{defi:cube}
	Let $k \ge 0$. We define corresponding cubes in the poset $\NN_{/[n]}$ by specifying the
	images of the vertices:
\begin{enumerate}
	\item $f: [1]^k \to \NN_{/[k]}$ with $f(j_1,\dots,j_k)$ given by
		\[
			[k] \to [k], \; i \mapsto i - 1 + j_i,
		\]
		where we set $j_0 = 1$,
	\item $b: [1]^k \to \NN_{/[k]}$ with $b(j_1,\dots,j_k)$ given by
		\[
			[k-1] \to [k], \; i \mapsto i + j_{i+1}
		\]
		so that $b = \partial_0^* f$, 
	\item and $q: [1]^{k+1} \to \NN_{/[k]}$ with
		\[
			q(j_0,\dots,j_k) = \begin{cases} b(j_1,\dots,j_k) & \text{if $j_0 = 0$,}\\
							f(j_1,\dots,j_k) & \text{if $j_0 = 1$.}
						\end{cases}
		\]
\end{enumerate}
\end{defi}

\begin{exa} For $k=2$, the cube $q$ may be depicted as
		\[
		\begin{tikzcd}[
		    ,column sep={2.5em,between origins}
		    ,row sep={2.5em,between origins}
		    ]
		    01 \arrow{rr}\arrow{dd}\arrow{dr} & & 02 \arrow{dd}\arrow{dr} & \\
		    & 001 \arrow[crossing over]{rr} & & 002 \arrow{dd} \\
		    11 \arrow{rr}\arrow{dr} & & 12 \arrow{dr} & \\
		    & 011 \arrow{rr}\arrow[<-,crossing over]{uu} & & 012.
		\end{tikzcd}
		\]
\end{exa}

\begin{defi}\label{defi:catnerve} Let $\B_{\bullet}: \NN^{\op} \to \St$ be a chain complex of stable
	$\infty$-categories with corresponding Grothendieck construction (\S
	\ref{subsec:grothendieck})
	\[
		\pi: \chi(\B_{\bullet}) \lra \N(\NN).
	\]
	We define, for every $n \ge 0$, the $\infty$-category
	\[
		\Nc(\B_{\bullet})_n \subset \Map_{\N(\NN)}(\N(\NN_{/[n]}), \chi(\B_{\bullet}))
	\]
	as the full subcategory spanned by the diagrams 
	\[
		A: \N(\NN_{/[n]}) \lra \chi(\B_{\bullet})	
	\]
	that satisfy the following conditions:
	\begin{enumerate}[label=(N\arabic*)]
		\item\label{N1} For every $k \ge 1$ and every degenerate $k$-simplex $\tau: [k] \to [n]$ of $\Delta^n$, the object
			$A_\tau$ is a zero object in the $\infty$-category $\B_k$.
		\item\label{N2} For every $k \ge 1$ and every nondegenerate $k$-simplex $\sigma: [k] \to [n]$, the corresponding
			cube 
			\[
				A|q^*\sigma: (\Delta^1)^{k+1} \lra \chi(\B_{\bullet}),
			\]
			obtained by restricting $A$ to the pullback of $\sigma$ along the canonical
			cube from Definition \ref{defi:cube}, is a $\pi$-limit diagram with limit
			vertex $(0,0,\dots,0)$.
	\end{enumerate}
	It follows from \cite[5.1.2.2]{lurie:htt} that the $\infty$-category
	$\Map_{\N(\NN)}(\N(\NN_{/[n]}), \chi(\B_{\bullet}))$ and, since limits commute with limits,
	the subcategory $\Nc(\B_{\bullet})_n$ is stable as well. The various stable
	$\infty$-categories $\Nc(\B_{\bullet})_n$ assemble to define a $2$-simplicial stable
	$\infty$-category with $2$-functoriality determined by \eqref{eq:2over}. We refer to 
	\[
		\Nc(\B_{\bullet}) \in \St_{\DDelta}
	\]
	as the {\em categorified Dold-Kan nerve} of $\B_{\bullet}$.
\end{defi}

To gain familiarity with the categorified Dold-Kan nerve, we provide an explicit description of the
low-dimensional simplices of $\Nc(\A_{\bullet})$ for a given chain complex $\B_{\bullet}$ of stable
$\infty$-categories:
\begin{enumerate}
		  \setcounter{enumi}{-1}
	\item We have $\Nc(\B_{\bullet})_0 \cong \B_0$.
	\item\label{list:1}  The $\infty$-category of $1$-simplices of $\Nc(\B_{\bullet})$ is equivalent to the
		$\infty$-category of diagrams of the form
		\[
		\begin{tikzcd}[
		    ,column sep={2.5em,between origins}
		    ,row sep={2.5em,between origins}
		    ]
		    A_0 \arrow{rr}\arrow{dd}& & A_1 \arrow{dd}\\
		    &&\\
		    A_{00} \arrow{rr}&& A_{01}
		\end{tikzcd}
		\]
		in $\chi(\B_{\bullet})$ where 
		\begin{enumerate}
			\item $\{A_i\}$ are objects of $\B_0$, $\{A_{ij}\}$ are
		objects of $\B_1$, 
	\item $A_{00}$ is a zero object in $\B_1$, 
			\item the square 
				\[
				\begin{tikzcd}[
				    ,column sep={2.5em,between origins}
				    ,row sep={2.5em,between origins}
				    ]
				    A_0 \arrow{rr}\arrow{dd}& & A_1 \arrow{dd}\\
				    &&\\
				    d(A_{00}) \arrow{rr}&& d(A_{01})
				\end{tikzcd}
				\]
				in $\B_0$ induced by $A$ is biCartesian so that it exhibits the object
				$d(A_{01})$ as the cofiber of the map $A_0 \to A_1$.
			\end{enumerate}
			\item The $\infty$-category $\Nc(\B_{\bullet})_2$ is equivalent to the
				$\infty$-category of diagrams in $\chi(\B_{\bullet})$ of the form
				\[
				\begin{tikzcd}[
				    ,column sep={2.5em,between origins}
				    ,row sep={2.5em,between origins}
				    ]
				    A_0 \arrow{rr}\arrow{dd}& & A_1 \arrow{dd}\arrow{rr} & & A_2 \arrow{dd} &\\
				    && && &\\
				    A_{00}\arrow{rr} & & A_{01} \arrow{rr}\arrow{dd}\arrow{dr} & & A_{02} \arrow{dd}\arrow{dr} & \\
								& & & A_{001} \arrow[crossing over]{rr} & & A_{002} \arrow{dd} \\
								& & A_{11} \arrow{rr}\arrow{dr} & & A_{12} \arrow{dr} & \\
								& & & A_{011} \arrow{rr}\arrow[<-,crossing over]{uu} & & A_{012}
				\end{tikzcd}
				\]
				with 
				\begin{enumerate}
					\item $\{A_i\}$ objects of $\B_0$, $\{A_{ij}\}$ objects of $\B_1$, $\{A_{ijk}\}$
						objects of $\B_2$,
					\item $A_{00}$ and $A_{11}$ are zero objects in $\B_1$, 
						$A_{001}$, $A_{011}$, and $A_{002}$ are zero objects in $\B_2$,
					\item the diagram $A$ exhibits the objects $d(A_{01})$, $d(A_{02})$, and $d(A_{12})$ as 
						cofibers of the morphisms $A_0 \to A_1$, $A_0 \to A_2$, and $A_1 \to
						A_2$, respectively, as detailed in \ref{list:1}. In particular, by the octahedral lemma, the
						square
							\[
							\begin{tikzcd}[
							    ,column sep={2.5em,between origins}
							    ,row sep={2.5em,between origins}
							    ]
							    d(A_{01}) \arrow{rr}\arrow{dd}& & d(A_{02}) \arrow{dd}\\
							    &&\\
							    d(A_{11}) \arrow{rr}&& d(A_{12})
							\end{tikzcd}
							\]
						is biCartesian.
					\item The cube
						\[
						\begin{tikzcd}[
						    ,column sep={2.5em,between origins}
						    ,row sep={2.5em,between origins}
						    ]
							A_{01} \arrow{rr}\arrow{dd}\arrow{dr} & & A_{02} \arrow{dd}\arrow{dr} & \\
							   & d(A_{001}) \arrow[crossing over]{rr} & & d(A_{002}) \arrow{dd} \\  
							   A_{11} \arrow{rr}\arrow{dr}& & A_{12}\arrow{dr} \\
							  & d(A_{011}) \arrow{rr}\arrow[<-,crossing over]{uu} & & d(A_{012})
						\end{tikzcd}
						\]
						in $\B_1$ induced by the diagram $A$ is biCartesian so that it exhibits the
						object $d(A_{012})$ as a totalization of the $3$-term complex
						$A_{01} \to A_{02} \to A_{12}$.
				\end{enumerate}

			\item[(n)] Similarly, the higher-dimensional simplices of
				$\Nc(\B_{\bullet})$ consist of collections of diagrams in $\B_k$
				parametrized by the various posets of $k$-simplices of $\Delta^n$,
				together with additional compatibility data that, for every nondegenerate
				$k$-simplex $\sigma$ in $\Delta^n$, exhibits the object
				$d(A_{\sigma})$ as the totalization of a $(k+1)$-term complex with
				underlying sequence of maps
				\[
					A_{\partial_k\sigma} \to A_{\partial_{k-1}\sigma} \to \dots
					\to A_{\partial_0\sigma}.
				\]

\end{enumerate}

Our next goal is to study special cases of the categorified Dold-Kan nerve and exhibit how they
relate to previously studied constructions. 

\begin{exa}\label{exa:relativeS} Let 
	\[
		\B_{\bullet} = \B_0 \overset{d}{\lla} \B_1 \lla 0 \lla \dots
	\]
	be a $2$-term chain complex of stable $\infty$-categories concentrated in degrees $0,1$.
	Then the $\infty$-category of $n$-simplices of $\Nc(\B_{\bullet})$ consists of diagrams
	\[
	\begin{tikzcd}
		A_0 \arrow{r}\arrow{d}& A_1 \arrow{d}\arrow{r} & A_2 \arrow{d}\arrow{r} &\dots \arrow{r} & A_n \arrow{d}\\
	    A_{00}\arrow{r} & A_{01} \arrow{r}\arrow{d} & A_{02} \arrow{d} \arrow{r} & \dots
	    \arrow{r} & A_{0n} \arrow{d}\\
	    & A_{11} \arrow{r} & A_{12}\arrow{d}\arrow{r} & \ddots  & \vdots\\
	    & & A_{22} \arrow{r}& \ddots & \vdots \arrow{d}\\
					&&&& A_{nn}
	\end{tikzcd}
	\]
	in $\chi(\B_{\bullet})$ where
	\begin{itemize}
		\item $\{A_i\}$ are objects of $\B_0$, $\{A_{ij}\}$ are objects of $\B_1$,
		\item the objects $\{A_{ii}\}$ are zero objects in $\B_1$,
		\item for every $0 \le i < j \le n$, the diagram
			\[
			\begin{tikzcd}
				A_i \arrow{r}\arrow{d}& A_j \arrow{d}\\
				d(A_{ii}) \arrow{r} & d(A_{ij})
			\end{tikzcd}
			\]
			in $\B_0$ induced by $A$ is biCartesian,
		\item for every $0 \le i < j < k \le n$, the diagram
			\[
			\begin{tikzcd}
				A_{ij} \arrow{r}\arrow{d}& A_{ik} \arrow{d}\\
				A_{jj} \arrow{r} & A_{jk}
			\end{tikzcd}
			\]
			in $\B_0$ induced by $A$ is biCartesian.
	\end{itemize}
	In this special case, the categorified Dold-Kan nerve $\Nc(\B_{\bullet})$ can thus be
	regarded as an $\infty$-categorical variant of Waldhausen's relative
	$S_{\bullet}$-construction associated to the functor $d: \B_1 \to \B_0$ of stable
	$\infty$-categories. In particular, for $\B_0 = 0$, we recover Waldhausen's
	$S_{\bullet}$-construction of the stable $\infty$-catgeory $\B_1$.
\end{exa}

\begin{exa} Let $\B$ be a stable $\infty$-category, $k \ge 0$, and let $\B[k]$ denote the chain
	complex of stable $\infty$-categories that has $\B$ in degree $k$ and the zero category
	$\Delta^0$ in all other degrees. We describe the categorified Dold-Kan nerve $\Nc(\B[k])$.
	Consider the cube
	\[
		b: [1]^{k+1} \lra \Fun([k],[k+1])
	\]
	from Definition \ref{defi:cube}. In dimensions $1$, $2$, and $3$, the respective
	cubes can be depicted as follows:
	\begin{center}
		\begin{minipage}{.3\textwidth}
			\[
			\begin{tikzcd}[
			    ,column sep={2.5em,between origins}
			    ,row sep={2.5em,between origins}
			    ]
			    {0} \arrow{rr} && {1}
			    \end{tikzcd}
			\]
		\end{minipage} 
		\begin{minipage}{.3\textwidth}
		\[
		\begin{tikzcd}[
		    ,column sep={2.5em,between origins}
		    ,row sep={2.5em,between origins}
		    ]
		    01 \arrow{rr}\arrow{dd}& & 02 \arrow{dd}\\
		    &&\\
		    11 \arrow{rr}&& 12 
		\end{tikzcd}
		\]
		\end{minipage} 
		\begin{minipage}{.3\textwidth}
		\[
		\begin{tikzcd}[
		    ,column sep={2.5em,between origins}
		    ,row sep={2.5em,between origins}
		    ]
		    012 \arrow{rr}\arrow{dd}\arrow{dr} & & 013 \arrow{dd}\arrow{dr} & \\
		    & 112 \arrow[crossing over]{rr} & & 113 \arrow{dd} \\
		    022 \arrow{rr}\arrow{dr} & & 023 \arrow{dr} & \\
		    & 122 \arrow{rr}\arrow[<-,crossing over]{uu} & & 123
		\end{tikzcd}
		\]
	\end{minipage}
\end{center}
	The $n$-simplices of $\Nc(\B[k])$ are then given by diagrams
	\[
		A: \N(\Fun([k],[n])) \lra \B
	\]
	subject to the following conditions:
	\begin{enumerate}
		\item For every degenerate $k$-simplex $\tau:[k] \to [n]$, the object $A_\tau$ is
			a zero object in $\B$.
		\item For every nondegenerate $(k+1)$-simplex $\sigma:[k+1] \to [n]$, the cube $(\sigma
			\circ b)^*A$
			is an exact cube in $\B$.
	\end{enumerate}
	As already explained in Example \ref{exa:relativeS}, $\Nc(\B[1])$ is an 
	$\infty$-categorical variant of Waldhausen's $S_{\bullet}$-construction. The simplicial
	object $\Nc(\B[2])$ can be regarded as an $\infty$-categorical version of the
	$S^{2,1}_{\bullet}$-construction due to Hesselholt-Madsen \cite{hesselholt-madsen}. For $k
	\ge 3$, the simplicial object underlying $\Nc(\B[k])$ is an $\infty$-categorical version of
	the $k$-dimensional $S^{\langle k \rangle}_{\bullet}$-construction as recently introduced
	and studied for abelian categories in \cite{poguntke:higher}. 
\end{exa}

\section{Proof of Theorem \ref{thm:catdk}}

Our strategy for the proof of Theorem \ref{thm:catdk} is to produce a step-by-step categorification
of the proof of the classical Dold-Kan correspondence presented in \S \ref{sec:dk}. 

\subsection{The equivalence $\Cc \circ \Nc \simeq \id$}

\begin{prop}\label{prop:compgroth} There is a natural weak equivalence
	\[
		\Cc \circ \Nc \overset{\simeq}{\lra} \Gamma \circ \chi
	\]
	where $(\Gamma, \chi)$ denote the functors comprising the Grothendieck construction from \S
	\ref{subsec:grothendieck}. 
\end{prop}
\begin{proof} Let $\B_{\bullet}$ be a chain complex of stable $\infty$-categories. For $n \ge 0$, the
	$\infty$-category $\Cc(\Nc(\B_{\bullet}))_n$ consists of diagrams
	\[
		A: \N(\NN_{/[n]}) \lra \chi(\B_{\bullet})
	\]
	subject to the conditions spelled out in Definition \ref{defi:catnerve} and additionally
	satisfying, for every $1 \le i \le n$,	
	\[
		d_i(A) \simeq 0.
	\]
	These conditions imply that the only nonzero objects in the diagram comprising $A$ are the
	ones parametrized by the $n$-simplex $\id: [n] \to [n]$ and the $(n-1)$-simplex
	$\partial_0: [n-1] \to [n]$. In particular, restriction along the functor
	\[
		\N(\NN_{/n}) \to \N(\NN_{/[n]}),
	\] 
	where $\NN_{/n}$ denotes the overcategory $0 \to 1 \to \dots \to n$ of $n$ in the poset
	$\NN$, induces an equivalence $\Cc(\Nc(\B_{\bullet}))_n \simeq \Gamma(\chi(\B_{\bullet}))_n$.
	Since this map is functorial in $n$ and $\B_{\bullet}$, we obtain a natural weak
	equivalence
	\[
		\Cc \circ \Nc \overset{\simeq}{\lra} \Gamma \circ \chi.
	\]
\end{proof}

\begin{cor} There is a natural equivalence $\Cc \circ \Nc \simeq \id$ as endofunctors of the
	$\infty$-category $\L \Ch_{\ge 0}(\St)$ of connective chain complexes of stable
	$\infty$-categories.
\end{cor}
\begin{proof} Immediate from Proposition \ref{prop:compgroth} and Proposition
	\ref{prop:groth}.
\end{proof}

\subsection{The equivalence $\Nc \circ \Cc \simeq \id$}

We proceed by showing $\Nc \circ \Cc \simeq \id$. To this end, we produce a zigzag diagram 
\begin{equation}\label{eq:nateq}
	\id \overset{\eta}{\lra} \GGamma \circ \CChi \overset{\alpha}{\lla} \Fc
	\overset{\beta}{\lra} \widetilde{\Nc} \circ \Cc
	\overset{\theta}{\lla} \Nc \circ \Cc
\end{equation}
of endofunctors of $\St_{\DDelta}$ and show that all maps in the diagram are weak equivalences.
The resulting equivalence $\id \simeq \Nc \circ \Cc$ of endofunctors of the
$\infty$-categorical localization $\L \St_{\DDelta}$ can be interpreted as a categorification of the unit
transformation $u: \id \to \N \circ \C$ constructed in the proof of Theorem \ref{thm:cdk}.

\subsubsection{The lax version $\widetilde{\Nc}$ of $\Nc$}

In the definition of the categorified Dold-Kan nerve of a connective chain complex $\B_{\bullet}$ of
stable $\infty$-categories, we have used the Grothendieck construction $\pi: \chi(\B_{\bullet}) \to
\N(\NN)$. We denote by $\widetilde{\Nc}(\B_{\bullet})$ the mutatis mutandis definition obtained by
using the lax Grothendieck construction $\pi: \CChi(\B_{\bullet}) \to \N(\NN)$ instead, where we interpret
$\NN$ as a $2$-category with discrete morphism categories. 

\begin{prop} There is a weak equivalence 
	\[
		\theta: \Nc \overset{\simeq}{\lra} \widetilde{\Nc}
	\]
	of functors $\Ch_{\ge 0}(\St) \to \St_{\DDelta}$.
\end{prop}
\begin{proof} This is an immediate consequence Proposition
	\ref{prop:laxbasic}\ref{prop:laxbasic:3}.
\end{proof}

\subsubsection{The functor $\Fc$}

To simplify notation, we introduce $\DDelta' = \DDelta^{(-,\op)}$. For $n \ge 0$, we introduce the
pushout
\[
	\M_n = \Nsc(\DDelta'_{/[n]}) \coprod_{\N(\NN_{/[n]})} \Delta^1 \times \N(\NN_{/[n]})
\] 
along the inclusion $\{0\} \times \id: \N(\NN_{/[n]}) \subset \Delta^1 \times
\N(\NN_{/[n]})$. We further denote the inclusions
\[
	\Nsc(\DDelta'_{/[n]}) \overset{r}{\lra} \M_n \overset{s}{\lla} \N(\NN_{/[n]})
\]
where $s = \{1\} \times \id$.
We introduce a functor 
\[
	\Fc: \St_{\DDelta} \lra \St_{\DDelta}
\]
as follows: Given a $2$-simplicial stable $\infty$-category $\A_{\bullet}$, we define
\[
	\Fc(\A_{\bullet})_n \subset \Map_{\Nsc(\DDelta')}(\M_n, \CChi(\A_{\bullet}))
\]
to be the full subcategory spanned by the diagrams
\[
	X:\; \M_n \lra \CChi(\A_{\bullet})
\]
satisfying the following conditions:
\begin{enumerate}[label=(F\arabic*)]
	\item \label{P1} The functor $r^*X: \Nsc(\DDelta'_{/[n]}) \lra \CChi(\A_{\bullet})$ maps edges
		corresponding to strictly commutative triangles to $\pi$-Cartesian edges in $\CChi(\A_{\bullet})$.
	\item \label{P2} The functor $s^*X$ maps every vertex of $\N(\NN_{/[n]})$ corresponding to a degenerate simplex
		$\tau:[k] \to [n]$ to a zero object in the fiber $\pi^{-1}([k])$.
	\item \label{P3} For every nondegenerate simplex $\sigma: [k] \hra [n]$, the composite
		\[
			\Delta^1 \times (\Delta^1)^k \overset{\id \times f^* \sigma}{\lra}
			\Delta^1 \times \N(\NN_{/[n]}) \overset{X}{\lra} \CChi(\A_{\bullet})
		\]
		is a biCartesian cube in the fiber $\pi^{-1}([k])$ where $f$ denotes the cube from
		Definition \ref{defi:cube}.
\end{enumerate}

\begin{prop} Restriction along $r: \N(\DDelta'_{/[n]}) \subset \M_n$ defines a functor 
	\[
		r^*: \Fc \lra \GGamma \circ \CChi 
	\]
	which is a weak equivalence. 
\end{prop}
\begin{proof}
	Let $\A_{\bullet} \in \St_{\DDelta}$, let $\pi:\CChi(\A_{\bullet}|\NN) \to \N(\NN)$
	denote the lax Grothendieck construction, and let $n \ge 0$. Let $K \subset \Delta^1 \times
	\N(\NN_{/[n]})$ be the full subcategory spanned by the vertices of $\{0\} \times
	\N(\NN_{/[n]})$ and those vertices of $\{1\} \times \N(\NN_{/[n]})$ that correspond to
	degenerate simplices $[k] \to [n]$ of $\Delta^n$. We then have the following statements:
	\begin{enumerate}
		\item \label{it:kan1} For a functor
			\[
				Y: K \lra \CChi(\A_{\bullet}|\NN),
			\]
			over $\N(\NN)$, the following conditions are equivalent:
			\begin{enumerate}
				\item For every degenerate simplex $[k] \to [n]$, the functor $Y$
					maps the corresponding vertex of $\{1\} \times
					\N(\NN_{/[n]})$ to a zero object in the fiber $\pi^{-1}([k])$.
				\item $Y$ is a $\pi$-right Kan extension of its restriction $Y|\{0\}
					\times \N(\NN_{/[n]})$.
			\end{enumerate}
		\item \label{it:kan2} For a functor
			\[
				Z: \Delta^{1} \times \N(\NN_{/[n]}) \to \CChi(\A_{\bullet}|\NN),
			\]
			over $\N(\NN)$, the following conditions are equivalent:
			\begin{enumerate}
				\item For every nondegenerate simplex $\sigma: [k] \hra [n]$, the composite
					\[
						\Delta^1 \times (\Delta^1)^k \overset{\id \times (\sigma \circ f)}{\lra}
						\Delta^1 \times \N(\NN_{/[n]}) \overset{Z}{\lra} \CChi(\A_{\bullet})
					\]
					is a biCartesian cube in the fiber $\pi^{-1}([k])$.
				\item $Z$ is a $\pi$-left Kan extension of its restriction $Z|K$.
			\end{enumerate}
	\end{enumerate}
	Now let $S \subset \Map_{\N(\NN)}(\Delta^{1} \times \N(\NN_{/[n]}),
	\CChi(\A_{\bullet}|\NN))$ denote the full subcategory spanned by those vertices
	that satisfy conditions \ref{P2} and \ref{P3}. Then, by \cite[4.3.1.15]{lurie:htt} and
	\ref{it:kan1}, \ref{it:kan2}, the restriction map
	\[
		S \to \Map_{\N(\NN)}(\Delta^{0} \times \N(\NN_{/[n]}), \CChi(\A_{\bullet}|\NN))
	\]
	is a trivial Kan fibration. Let $M \subset \Map_{\Nsc(\DDelta')}(\N(\DDelta'_{/[n]}),
	\CChi(\A_{\bullet}))$ denote the full subcategory spanned by the vertices that satisfy
	condition \ref{P1}. By definition, we have $M = (\GGamma \circ \CChi(\A_{\bullet}))_n$. 
	Then we have a pullback diagram of simplicial sets
	\[
		\begin{tikzcd}
			\Fc(\A_{\bullet})_n \arrow{r}\arrow{d} & S\arrow{d}\\
			M \arrow{r} & \Map_{\N(\NN)}(\Delta^{0} \times \N(\NN_{/[n]}),\CChi(\A_{\bullet}|\NN))
		\end{tikzcd}
	\]
	so that $\Fc(\A_{\bullet})_n \to (\GGamma \circ \CChi(\A_{\bullet}))_n$ is a trivial Kan
	fibration and hence an equivalence of $\infty$-categories.
\end{proof}

\begin{prop}\label{prop:main}
	Restriction along $s: \N(\NN_{/[n]}) \hra \M_n$ defines a natural transformation
	\[
		s^*: \Fc \lra \widetilde{\Nc} \circ \Cc
	\]
	of endofunctors of $\St_{\DDelta}$ which is a pointwise weak equivalence.
\end{prop}
\begin{proof}
	Let $\A_{\bullet} \in \St_{\DDelta}$ and $n \ge 0$. Let $X \in \Fc(\A_{\bullet})_n$ and let
	\[
		A: \N(\NN_{/[n]}) \lra \CChi(\A_{\bullet})
	\]
	be its restriction along $s$. 
	We show the following:
	\begin{enumerate}
		\item\label{main:1} For $0 < i \le k$, let 
			\[
				\partial_i^*: \CChi(\A_{\bullet})_k \to \CChi(\A_{\bullet})_{k-1}
			\]
			denote a functor obtained by lifting the morphism $\partial_i: [k-1] \to [k]$ of
			$\DDelta'$ with respect to the locally Cartesian fibration
			$\CChi(\A_{\bullet}) \to \N(\DDelta')$. Then, for every $\sigma: [k] \to
			[n]$, we have $\partial_i^*(A_{\sigma}) \simeq 0$.
		\item\label{main:2} For every $k \ge 0$ and every nondegenerate $(k+1)$-simplex
			$\sigma: [k+1] \to [n]$, the corresponding cube 
			\[
				A|q^*\sigma: (\Delta^1)^{k+2} \lra \CChi(\A_{\bullet})|\N(\NN),
			\]
			obtained by restricting $A$ to the pullback of $\sigma$ along the 
			cube from Definition \ref{defi:cube}, is a $\pi$-limit diagram with limit
			vertex $(0,0,\dots,0)$.
	\end{enumerate}
	To verify \ref{main:1}, we first note that, by definition, for every degenerate simplex $[k] \to [n]$, the
	corresponding object $A_{\tau}$ is a zero object so that there is nothing to show. The value
	of $A$ at a nondegenerate simplex $\sigma: [k] \hra [n]$ is given by the totalization of the
	cube $X|f^*\sigma$. Due to condition \ref{P1}, this cube has the property that every edge
	parallel to the $i$th coordinate axis of the cube gets mapped under $\partial_i^*$ to an
	equivalence in $\CChi(\A_{\bullet})_{k-1}$. Since totalization commutes with the functor
	$\partial_i^*$, it follows that the totalization of $X|f^*\sigma$ is zero by Lemma
	\ref{lem:cube} showing \ref{main:1}.

	We prove \ref{main:2}. Let $\pi: \CChi(\A_{\bullet})|\N(\NN) \to \N(\NN)$ denote the
	Cartesian fibration obtained by restricting $\CChi(\A_{\bullet})$. 
	We first show the following claim: 
	\begin{enumerate}[label=(\Roman*)]
		\item \label{C} For every nondegenerate simplex $\sigma: [k] \hra
		[n]$, the cube $X|\Delta^1 \times q^*\sigma$ is a $\pi$-limit cube. 
	\end{enumerate}
	Note that \ref{P3}
	implies that the front face $X|\Delta^1 \times f^*\sigma$ of this cube is biCartesian in the
	fiber $\CChi(\A_{\bullet})_{[k]}$ so
	that it suffices to show that the back face $C := X|\Delta^1 \times b^*\sigma$ is biCartesian in
	the fiber $\CChi(\A_{\bullet})_{[k-1]}$. Again by property \ref{P3}, the face $F := X|\Delta^1
	\times f^*(\sigma \circ \partial_0)$ of $C$ is biCartesian in $\CChi(\A_{\bullet})_{[k-1]}$.
	We need to show that the face of $C$ opposite to $F$ is biCartesian as well. To see this, we
	argue as follows: Consider the restriction $R$ of $X$ to the inclusion $\Delta^1 \times
	\N(\Fun([k-1],[n])) \subset \Delta^1 \times \N(\DDelta'_{/[n]})$. Let $K \subset \Delta^1 \times
	\N(\Fun([k-1],[n]))$ denote the nerve of the poset spanned by $\{0\} \times \Fun([k-1],[n])$
	together with all elements of $\{1\} \times \Fun([k-1],[n])$ whose second component is a
	degenerate simplex. Then property \ref{P3}, applied to the simplices $\sigma \circ
	\partial_i$, $i > 0$, implies that $X|R$ is a left Kan extension of $X|K$ where we consider
	both functors with values in the fiber $\CChi(\A_{\bullet})_{[k-1]}$. Now let $K' \subset R$
	denote the nerve of the subposet spanned by all elements except $\{1\} \times 
	\sigma \circ \partial_1$. The fact that $X|R$ is a left Kan extension of $X|K'$ translates via the pointwise
	formula for Kan extensions to the statement that the face of $C$ opposite of $F$ is
	biCartesian, proving the claim \ref{C}. 

	Now \ref{C} implies our desired statement as follows: The cube $A|q^*\sigma$ is the face of the
	larger cube $X|\Delta^1 \times q^*\sigma$ obtained by restriction along $\{1\} \times q^*
	\sigma$. Since this cube is a $\pi$-limit cube, it suffices to show that the cube $X|\{0\}
	\times q^*\sigma$ is a $\pi$-limit. But this is clear, since all edges of the arrow
	$X|\{0\} \times b^*\sigma \to X|\{0\} \times f^*\sigma$ (which comprises the cube $X|\{0\}
	\times q^*\sigma$) are $\pi$-Cartesian. 

	By Proposition \ref{prop:catconservative} below, to show that $s^*$ is a weak equivalence,
	it suffices to verify that $\Cc(s^*)$ is a weak equivalence. This is easily seen by direct
	inspection.
\end{proof}

\begin{lem}\label{lem:cube}
	Let $\A$ be a stable $\infty$-category and let $C: (\Delta^1)^k \to \A$ be a cube in $\A$.
	Let $B$ and $F$ denote the restrictions of the cube $C$ to $\{0\} \times
	(\Delta^1)^{k-1}$ and $\{1\} \times (\Delta^1)^{k-1}$, respectively. Then the following are
	equivalent:
	\begin{enumerate}
		\item $\tot(C)$ and $\tot(B)$ are zero objects in $\A$.
		\item $\tot(C)$ and $\tot(F)$ are zero objects in $\A$.
		\item $\tot(F)$ and $\tot(B)$ are zero objects in $\A$.
	\end{enumerate}
\end{lem}
\begin{proof}
	This follows immediately from the fact that there exists an exact triangle
	\[
		\begin{tikzcd}
			\tot(B) \arrow{d}\arrow{r} & \tot(F) \arrow{d}\\
			 0 \arrow{r} & \tot(C)
		\end{tikzcd}
	\]
	in $\A$.
\end{proof}

Collecting all results of this section, we obtain the following main result:

\begin{thm} There is a natural equivalence
	\[
		\id \overset{\simeq}{\lra} \Nc \circ \Cc
	\]
	of endofunctors of the $\infty$-category $\L \St_{\DDelta}$ of $2$-simplicial stable
	$\infty$-categories.
\end{thm}
\begin{proof}
	The various results of this section imply the existence of a diagram of natural weak equivalences
	\[
		\id \lra \GGamma \circ \CChi \lla \Fc \lra \widetilde{\Nc} \circ \Cc \lla \Nc \circ
		\Cc
	\]
	which leads to the desired conclusion after localization.
\end{proof}

\subsubsection{The functor $\Cc$ is conservative}
\label{sec:conservative}

\begin{prop}\label{prop:catconservative} The categorified normalized chains functor
	\[
		\Cc: \St_{\DDelta} \lra \Ch_{\ge 0}(\St)
	\]
	is conservative: a morphism $f$ in $\St_{\DDelta}$ is a weak equivalence if and only if
	$\Cc(f)$ is a weak equivalence.
\end{prop}
\begin{proof}
	The proof is a step-by-step categorification of the proof of Proposition \ref{prop:conservative}.
	Given a $2$-simplicial stable $\infty$-category $\A_{\bullet}$, we introduce its path object
	$\Pc(\A_{\bullet})$, which is the $2$-simplicial object obtained by pullback along the
	$2$-functor 
	\[
		\DDelta \lra \DDelta, \; [n] \mapsto [n] \ast [0].
	\]
	The values of the path object are given by $\Pc(\A_{\bullet})_n = \A_{n+1}$. The various
	omitted face maps $d_n: \A_n \to \A_{n-1}$ define a natural map of $2$-simplicial stable
	$\infty$-categories $d: \Pc(\A_{\bullet}) \to \A_{\bullet}$. For every $n \ge 0$, we denote
	by $\Omega(\A_{\bullet})_n$ the full subcategory of $\A_{n+1}$ spanned by the objects $X$ such
	that $d_{n+1}(X)$ is a zero object in $\A_{n}$. We obtain a $2$-simplicial stable
	$\infty$-category $\Omega(\A_{\bullet})$ which is part of a sequence
	\[
		\Omega(\A_{\bullet}) \hra \Pc(\A_{\bullet}) \overset{d}{\lra} \A_{\bullet}
	\]
	in $\St_{\DDelta}$, functorial in $\A_{\bullet}$, with composite equivalent to the zero map.

	For a connective chain complex $\B_{\bullet}$ of stable $\infty$-categories, we define
	$\Omega(\B_{\bullet})$ as the shifted chain complex
	\[
		\B_1 \lla \B_2 \lla \cdots
	\]
	omitting $\B_0$. It is immediate from the definitions that we have an equality
	\begin{equation}\label{eq:oc}
			\Cc \circ \Omega = \Omega \circ \Cc
	\end{equation}
	of functors from $\St_{\DDelta}$ to $\Ch_{\ge 0}(\St)$. 

	We now proceed by showing the following statement by induction on $n$:
	\begin{enumerate}[label=(\Roman{*})]
		\item\label{it:I} Let $n \ge 0$. Then, for every map $f: \A_{\bullet} \to \A'_{\bullet}$ of
			$2$-simplicial stable $\infty$-categories, such that $\Cc(f)$ is a weak
			equivalence, the map $f_n: \A_n \to \A'_n$ is an equivalence of stable
			$\infty$-categories.
	\end{enumerate}
	The statement is obvious for $n = 0$, since $\Cc(f)_0 = f_0$. Assume that \ref{it:I} holds
	for a fixed $n \ge 0$. Given a map $\A_{\bullet} \to \A'_{\bullet}$, we consider the
	commutative diagram
	\[
	\begin{tikzcd}
		\Omega(\A_{\bullet}) \arrow{r}\arrow{d}{\Omega(f)} & \Pc(\A_{\bullet})
		\arrow{r}\arrow{d}{\Pc(f)} &\A_{\bullet} \arrow{d}{f}\\ 
		\Omega(\A'_{\bullet}) \arrow{r} & \Pc(\A'_{\bullet}) \arrow{r} &\A'_{\bullet}
	\end{tikzcd}
	\]
	in $\St_{\DDelta}$. Evaluating the diagram at $[n] \in \DDelta$, we obtain the diagram
	\begin{equation}\label{eq:nat1}
			\begin{tikzcd}
				\Omega(\A_{\bullet})_n \arrow{r}\arrow{d}{\Omega(f)_n} & \A_{n+1}
				\arrow{r}{d_{n+1}}\arrow{d}{f_{n+1}} &\A_{n} \arrow{d}{f_n}\\ 
				\Omega(\A'_{\bullet})_n \arrow{r} & \A'_{n+1} \arrow{r}{d_{n+1}} &\A'_{n}
			\end{tikzcd}
	\end{equation}
	of stable $\infty$-categories.
	By induction hypothesis, the functor $f_n$ is an equivalence.
	Further, by \eqref{eq:oc}, we have that $\Cc(\Omega(f)) = \Omega(\Cc(f))$ is a weak equivalence so
	that, again by induction hypothesis, the functor $\Omega(f)_n$ is an equivalences. 
	Note that the right square in \eqref{eq:nat1} can be completed to a commutative diagram
	\begin{equation}\label{eq:semidiag}
			\begin{tikzcd}
				\A_{n+1} \arrow[shift left=1.5ex]{r}{d_n}\arrow[shift
				right=1.5ex,swap]{r}{d_{n+1}}\arrow[swap]{d}{f_{n+1}} &\A_{n}
				\arrow{d}{f_n} \ar["s_{n}" description]{l}\\ 
				\A'_{n+1} \arrow[shift left=1.5ex]{r}{d_n}\arrow[shift right=1.5ex,swap]{r}{d_{n+1}}
				&\A'_{n}\arrow{l}[description]{s_{n}}
			\end{tikzcd}
	\end{equation}
	where $d_n$ and $s_n$ denote the respective face and degeneracy maps and we leave the
	$2$-categorical data implicit.  We deduce that $f_{n+1}$ is an equivalence by Lemma
	\ref{lem:semiglue}. This concludes the proof of \ref{it:I} and the lemma.
\end{proof}

\begin{lem}\label{lem:semiglue}
	Consider a diagram 
	\begin{equation}\label{eq:semimap}
		\begin{tikzcd}
			\X \arrow[shift left=1.5ex]{r}{q}\arrow[shift right=1.5ex,swap]{r}{p}\arrow[swap]{d}{f} & \A
			\arrow{d}{g} \ar["s" description]{l}\\ 
			\X' \arrow[shift left=1.5ex]{r}{q'}\arrow[shift right=1.5ex,swap]{r}{p'}
			&\A' \arrow{l}[description]{s'}
		\end{tikzcd}
	\end{equation}
	of stable $\infty$-categories with $sp = sq = \id_{\A}$, $s'p' = s'q' = \id_{\A'}$ so that
	these identities are counits and units, respectively, of adjunctions
	\[
		p \dashv s \dashv q
	\]
	and 
	\[
		p' \dashv s' \dashv q'.
	\]
	Set $\B = \ker(q)$ and $\B' = \ker(q')$. Suppose that the induced functors $f: \A \to \A'$
	and $\overline{g}: \B \to \B'$ are equivalences. Then the functor $g$ is an
	equivalence.
\end{lem}
\begin{proof}
	Consider the relative nerve $\pi: \N_s(\Delta^1) \lra \Delta^1$ of the functor $\Delta^1 \to
	\Cat_{\infty}$ determined by $s: \A \to \X$ (cf. \cite[3.2.5.12]{lurie:htt}). Since $s$ has
	the right adjoint $q$, the coCartesian fibration $\pi$ is Cartesian as well so that we have
	an equivalence $\X \simeq \Map^{\#}_{\Delta^1}(\Delta^1, \N_s(\Delta^1))$ with the
	$\infty$-category of Cartesian sections of $\N_s(\Delta^1)$. The latter $\infty$-category
	can be identified with the full subcategory of $\Fun(\Delta^1, \X)$ spanned by the counit
	edges, i.e., edges equivalent to $s(q(X)) \to X$. Here, an edge $e$ is a counit edge if and
	only if $q(e)$ is an equivalence in $\A$. But this is in turn equivalent to the statement
	that the cofiber of $e$ lies in $\B = \ker(q)$.  Consider the $\infty$-category $\X(\B,\A)$
	of diagrams 
	\[
		\begin{tikzcd} 
			A \arrow{r}\arrow{d} & X \arrow{r}\arrow{d} & 0\arrow{d}\\
			0 \arrow{r} & B \arrow{r} & A'
		\end{tikzcd}
	\]
	in $\X$ where $A \in \A$, $B \in \B$, and both squares are biCartesian (which implies 
	$A' \in \A$). The above discussion implies that the evaluation map at $X$ establishes an
	equivalence of $\infty$-categories $\X(\B,\A) \simeq \X$. Let $\Map_{\X}(\B,\A) \subset
	\Fun(\Delta^1, \X)$ denote the full subcategory spanned by those edges $e$ in $\X$ so that
	$d_1(e)$ is a vertex in $\B$ and $d_0(e)$ is a vertex in $\A$. Clearly, projection onto the
	bottom right edge provides an equivalence
	\[
		\X(\B,\A) \overset{\simeq}{\to} \Map_{\X}(\B,\A).
	\]

	We now construct an $\infty$-category $\Y$ equipped with a map $\theta: \Y \to \Delta^1$ as
	follows: an $n$-simplex in $\Y$ consists of 
	\begin{enumerate}[label=(\roman{*})]
		\item a map $f: [n] \to [1]$ in $\Delta$, 
		\item an $n$-simplex $\sigma: \Delta^n \to \X$ such that $\sigma|\Delta^{f^{-1}(0)}
			\subset \B$ and $\sigma|\Delta^{f^{-1}(1)} \subset \A$.
	\end{enumerate}
	Note that $\Map_{\X}(\B,\A)$ can be identified with the $\infty$-category of sections of
	$\theta$. The assumption that $p$ is a left adjoint to $s$ implies that the map $\theta$ is a
	coCartesian fibration where a section $e$ is a coCartesian edge if and only if $p(e)$ is an
	equivalence in $\A$.

	We have thus produced a diagram of equivalences of $\infty$-categories
	\[
		\begin{tikzcd}
			\X & \arrow{l} \X(\B,\A) \arrow{r} & \Map_{\X}(\B,\A) \arrow{r} & 
			\Map^{\#}_{\Delta^1}((\Delta^1)^{\flat}, \Y)
		\end{tikzcd}
	\]
	The diagram \eqref{eq:semimap} induces a map $t: \Y \to \Y'$ that preserves coCartesian edges
	and is a fiberwise equivalence. By \cite[3.3.1.5]{lurie:htt}, it follows that $t$ itself is
	an equivalence.  We conclude by noting the commutative diagram
	\[
		\begin{tikzcd}
			\X\arrow{d}{f}  & \arrow{l} \X(\B,\A) \arrow{r}\arrow{d}  & \Map_{\X}(\B,\A) \arrow{r}\arrow{d}  & 
			\Map^{\#}_{\Delta^1}((\Delta^1)^{\flat}, \Y) \arrow{d}{t} \\
			\X' & \arrow{l} \X'(\B',\A') \arrow{r} & \Map_{\X'}(\B',\A') \arrow{r} & 
			\Map^{\#}_{\Delta^1}((\Delta^1)^{\flat}, \Y').
		\end{tikzcd}
	\]
	where all horizontal arrows are equivalences and, since $t$ is an equivalence, the rightmost
	vertical map is an equivalence. By the two-out-of-three property the leftmost arrow $f$ is an equivalence as well.  
\end{proof}

\begin{rem} The equivalence 
	\begin{equation}\label{eq:semirec}
			\X \simeq \Map^{\#}_{\Delta^1}((\Delta^1)^{\flat}, \Y)
	\end{equation}
	appearing in the proof of Lemma \ref{lem:semiglue} admits the following interpretation: the
	stable $\infty$-category $\X$ comes equipped with a {\em semiorthogonal decomposition} $\X =
	\langle \B, \A \rangle$ satisfying a certain admissibility condition. In this situation, the
	equivalence \eqref{eq:semirec} shows that the $\infty$-category $\X$ can be recovered from
	the two components $\B$ and $\A$ of the decomposition together with the {\em gluing functor}
	$\overline{p}: \B \to \A$, using the terminology of \cite{bk:semi}.
\end{rem}

\bibliographystyle{alpha} 
\bibliography{refs} 

\begin{thebibliography}{{Lur}09b}

\bibitem[BK89]{bk:semi}
A.~Bondal and M.~Kapranov.
\newblock {Representable functors, Serre functors, and mutations}.
\newblock {\em Izvestiya Rossiiskoi Akademii Nauk. Seriya Matematicheskaya},
  53(6):1183--1205, 1989.

\bibitem[DJ17]{dj:sk}
T.~Dyckerhoff and G.~Jasso.
\newblock {Higher-dimensional S-constructions and higher-dimensional
  Auslander-Reiten theory}.
\newblock {\em in preparation}, 2017.

\bibitem[DK17]{dk:triangulated}
T.~Dyckerhoff and M.~Kapranov.
\newblock Triangulated surfaces in triangulated categories.
\newblock {\em to appear in JEMS}, 2017.

\bibitem[DKSS17]{dkss:schober}
T.~Dyckerhoff, M.~Kapranov, V.~Schechtman, and Y.~Soibelman.
\newblock {Topological Fukaya categories with coefficients}.
\newblock {\em in preparation}, 2017.

\bibitem[Dol58]{dold:homology}
A.~Dold.
\newblock {Homology of symmetric products and other functors of complexes}.
\newblock {\em Annals of Mathematics}, pages 54--80, 1958.

\bibitem[Dyc17a]{d:grothendieck}
T.~Dyckerhoff.
\newblock {A lax Grothendieck construction}.
\newblock {\em in preparation}, 2017.

\bibitem[Dyc17b]{d:a1homotopy}
T.~Dyckerhoff.
\newblock {{$\Bbb{A}^1$} -homotopy invariants of topological {F}ukaya
  categories of surfaces}.
\newblock {\em Compos. Math.}, 153(8):1673--1705, 2017.

\bibitem[HM13]{hesselholt-madsen}
L.~Hesselholt and I.~Madsen.
\newblock Real algebraic {K}-theory.
\newblock {\em to appear}, 2013.

\bibitem[Iya07]{iyama:higher}
O.~Iyama.
\newblock {Higher-dimensional Auslander--Reiten theory on maximal orthogonal
  subcategories}.
\newblock {\em Advances in Mathematics}, 210(1):22--50, 2007.

\bibitem[Kan58]{kan:functors}
Daniel~M Kan.
\newblock {Functors involving css complexes}.
\newblock {\em Transactions of the American Mathematical Society},
  87(2):330--346, 1958.

\bibitem[KS14]{ks:schobers}
M.~Kapranov and V.~Schechtman.
\newblock {Perverse schobers}.
\newblock {\em arXiv preprint arXiv:1411.2772}, 2014.

\bibitem[Lur09a]{lurie:htt}
J.~Lurie.
\newblock {\em Higher topos theory}, volume 170 of {\em Annals of Mathematics
  Studies}.
\newblock Princeton University Press, Princeton, NJ, 2009.

\bibitem[{Lur}09b]{lurie:2cat}
J.~{Lurie}.
\newblock {(Infinity,2)-Categories and the Goodwillie Calculus I}.
\newblock {\em ArXiv e-prints}, May 2009.

\bibitem[{Lur}11]{lurie:ha}
J.~{Lurie}.
\newblock {Higher Algebra}.
\newblock {\em preprint}, May 2011.

\bibitem[Lur14]{lurie:ktheory}
J.~Lurie.
\newblock {Algebraic K-Theory and Manifold Topology}.
\newblock {\em Lecture notes available the author's website}, 2014.

\bibitem[Pog17]{poguntke:higher}
T.~Poguntke.
\newblock {Higher Segal structures in algebraic $ K $-theory}.
\newblock {\em arXiv preprint arXiv:1709.06510}, 2017.

\bibitem[Wal85]{waldhausen}
F.~Waldhausen.
\newblock Algebraic {$K$}-theory of spaces.
\newblock In {\em Algebraic and geometric topology ({N}ew {B}runswick,
  {N}.{J}., 1983)}, volume 1126 of {\em Lecture Notes in Math.}, pages
  318--419. Springer, Berlin, 1985.

\end{thebibliography}

\end{document}